\documentclass[a4paper,12pt]{article}

\usepackage{ucs}
\usepackage{amsfonts} 
\usepackage{amsmath} 
\usepackage{graphicx,subfigure}
\usepackage{caption} 
\usepackage{enumerate}
\usepackage{MnSymbol}
\usepackage{color}
\usepackage{xfrac}

\title{}
\author{}
\date{}

\newtheorem{theorem}{Theorem}[section]

\newtheorem{corollary}[theorem]{Corollary}

\newenvironment{proof}[1][Proof]{\begin{trivlist}
\item[\hskip \labelsep {\bfseries #1}]}{\end{trivlist}}
\newenvironment{definition}[1][Definition]{\begin{trivlist}
\item[\hskip \labelsep {\bfseries #1}]}{\end{trivlist}}

\newenvironment{@abssec}[1]{%
\if@twocolumn
\section*{#1}%
\else
\vspace{.05in}\footnotesize
\parindent .2in
{\upshape\bfseries #1. }\ignorespaces
\fi}
{\if@twocolumn\else\par\vspace{.1in}\fi}

\newcommand{\qed}{\nobreak \ifvmode \relax \else
\ifdim\lastskip<1.5em \hskip-\lastskip
\hskip1.5em plus0em minus0.5em \fi \nobreak
\vrule height0.75em width0.5em depth0.25em\fi}

\newcommand\keywordsname{Key words}
\newcommand\AMSname{AMS subject classifications}

\newcommand{\Q}{\mbox{WB}}
\newcommand{\WB}{\mbox{WB}}

\newcommand{\torus}{{\mathbb{T}^d}}

\begin{document}
\title{Quantitative Quasiperiodicity}
\author{Suddhasattwa Das\footnotemark[1], \and Yoshitaka Saiki\footnotemark[2],
\and Evelyn Sander\footnotemark[4], \and James A Yorke\footnotemark[3]}
\footnotetext[1]{Department of Mathematics, University of Maryland, College Park}
\footnotetext[2]{Graduate School of Commerce and Management, Hitotsubashi University}
\footnotetext[3]{University of Maryland, College Park}
\footnotetext[4]{Department of Mathematical Sciences, George Mason University}

\date{\today}
\maketitle

\renewcommand*\contentsname{Contents}

\begin{abstract}
{ The Birkhoff Ergodic Theorem concludes that time averages, that is,
Birkhoff averages, $\Sigma_{n=1}^N f(x_n)/N$ of a function $f$ along an
ergodic trajectory $(x_n)$ of a function $T$ converges to the space
average $\int f d\mu$, where $\mu$ is the unique invariant probability
measure. 
Convergence of the time average to the space average is slow.
We introduce a modified average of $f(x_n)$ by giving very small weights to the
``end'' terms when $n$ is near $0$ or $N$. 
When $(x_n)$ is a trajectory
on a quasiperiodic torus and $f$ and $T$ are $C^\infty$, we show that 
our weighted Birkhoff averages converge ``super'' fast to $\int f d\mu$,
{\em i.e.} with error smaller than every polynomial of $1/N$. Our goal
is to show that our weighted Birkhoff average is a powerful
computational tool, and this paper illustrates its use for several
examples where the quasiperiodic set is one or two dimensional. In
particular, we compute rotation numbers and conjugacies (i.e. changes
of variables) and their Fourier series, often with 30-digit precision.}
\end{abstract}

\section{Introduction} 

Quasiperiodicity is a key type of observed
dynamical behavior in a diverse set of applications. Tori with
quasiperiodic motion persist for small perturbations by the
Kolmogorov-Arnold-Moser theory, but such behavior is also observed for
non-conservative systems well beyond this restricted regime. We believe
that quasiperiodicity is one of only three types of dynamical behaviors
occurring in basic sets of typical systems. See~\cite{sander:yorke:15}
for the statement of our formal conjecture of this basic set
triumvirate. For example, quasiperiodicity occurs in a system of weakly
coupled oscillators, in which there is an invariant smooth attracting
torus in phase space with behavior described exclusively by the phase
angles of rotation of the system. Indeed, it is the property of the
motion being described using only a set of phase angles that always
characterizes quasiperiodic behavior. In a now classical set of papers,
Newhouse, Ruelle, and Takens demonstrated a route to chaos through a
region with quasiperiodic behavior, causing a surge in the study of the
motion~\cite{newhouse:ruelle:takens:78}. There is active current
interest in development of a systematic numerical and theoretical
approach to bifurcation theory for quasiperiodic systems. 

Our goal in this paper is to present a fast numerical method for the
fast calculation of the limit of Birkhoff averages in quasiperiodic systems, allowing us to compute
various key quantities. If $f$ is
integrable and the dynamical system is ergodic on the set in which the trajectory lives, then the Birkhoff Ergodic Theorem asserts that the Birkhoff average $ \Sigma_{n=1}^N f(x_n)/N $ of a function $f$ along an ergodic
trajectory $(x_n)$ converges to the space average $\int f d\mu$ as
$N\to\infty$ for $\mu$-almost every $x_0$, where $\mu$ is the unique invariant probability measure.
We develop a numerical method for calculating the limit of such averages, where
instead of weighting the terms $f(x_n)$ in the average equally, we
weight the early and late terms of the set $\{1,\dots,N\}$ much less
than the terms with $n\sim N/2$ in the middle. That is, rather than
using the equal weighting $(1/N)$ in the Birkhoff average, we use a
weighting function $w(n/N)$, which will primarily be the following well
known $C^\infty$ function that we will call the {\bf exponential
weighting function}, $w_{\exp}(t) = \exp(1/(t(t-1))$. In a companion
paper~\cite{Das-Yorke}, it is rigorously shown that for quasiperiodic systems and $C^\infty$ $f$, this weighting
function leads to super convergence with respect to $N$, meaning faster
than any polynomial in $N^{-1}$. This super convergence arises from the
fact that we are taking advantage of the quasiperiodic nature of the map
or flow. In particular, our method uses the underlying structure of a
quasiperiodic system, and would not give improved convergence results
for chaotic systems. We demonstrate the method and its convergence rate
by computing rotation numbers, conjugacies, and their Fourier series in
dimensions one and two. We will refer to a 1D quasiperiodic curve as a
\emph{(topological) circle}.

Other authors have considered related numerical methods before,
in particular~\cite{seara:villanueva:06, luque:villanueva:14}, which we will
compare to when we introduce our method. See also~\cite{Simo1, Simo2,
Durand:2002ug,Baesens:1990vj,Broer:1993uv,Vitolo:2011it,Hanssmann:2012jy,
Sevryuk:2012dv,Broer:1990ip,Kuznetsov:2015dl}.

Our paper proceeds as follows: In Section~\ref{sec:Quasi}, we give the formal
definition of quasiperiodicity, rotation, and the conjugacy map. 
In Section~\ref{sec:Method_Rot}, we describe our
numerical method in detail. We illustrate our method for 
a series of four examples, including an example of a two-dimensionally quasiperiodic map. 
In all cases, we get fast convergence and are in most cases able to 
give results with thirty digits of precision. For convenience of the reader, 
have summarized our numerical findings in Table~\ref{table:summary}. 

We start by describing our results for a key example of
quasiperiodicity: the (circular) restricted three-body problem. This is an
idealized model of the motion of the planet, a large moon, and a spacecraft
governed by Newtonian mechanics, in a model studied by
Poincar{\'e}~\cite{ThreeBody2,ThreeBody1}. In particular, we consider a
planar three-body problem consisting of two massive bodies (``planet" and
``moon") moving in circles about their center of mass and a third body
(``spacecraft") whose mass is infinitesimal, having no effect on the
dynamics of the other two.

We assume that the moon has mass $\mu$ and the planet mass is $1-\mu$ where $\mu = 0.1$, and
writing equations in rotating coordinates around the center of mass. Thus the planet remains fixed at
$(-0.1,0)$, and the moon is fixed at $(0.9,0)$. In these coordinates, the satellite's 
location and velocity are given by the {\em generalized position vector} $(q_1,q_2)$ 
and {\em generalized velocity vector}
$(p_1,p_2)$. 
The equations of motion are as follows (see \cite{ThreeBody1}). 
\begin{equation}
{\displaystyle
\begin{array}{rcl}\label{eqn:ThreeBody}
dq_1/dt &=& p_1+q_2, \\
dq_2/dt &=& p_2-q_1, \\
dp_1/dt &=& p_2- \mu(q_1-1+\rho) d_{moon}^{-3} -(1-\mu)(q_1+ \mu) d_{planet}^{-3},\\
dp_2/dt &=& -p_1-\mu q_2 d_{moon}^{-3}-(1-\mu)q_2 d_{planet}^{-3},
\end{array}
}
\end{equation}
where 
\[ d_{moon} = ((q_1-1+\mu)^2+q_2^2)^{0.5} \mbox{ and } d_{planet} = ((q_1+\mu)^2+q_2^2)^{0.5}.
\]
The following function $H$ is a Hamiltonian for this system
\begin{equation}\label{eqn:Hamiltonian}
H=[(p_1^2+p_2^2)/2]+[q_2p_1-q_1p_2]-[\mu~d_{planet}^{-1} + (1-\mu)~d_{moon}^{-1}].
\end{equation}
The terms in the square brackets are resp. the kinetic energy, angular moment and the angular potential.
For fixed $H$, Poincar{\'e} reduced this problem to the study of the Poincar{\'e} return map for a fixed value of $H$, only considering 
a discrete trajectory of the values of $(q_1,p_1)$ on the section $q_2=0$ and $\frac{dq_2}{dt}>0$. 
Thus we consider a map in two dimensions rather than a flow in four dimensions. 
Figure \ref{fig:3B3} shows one possible motion of the spacecraft for the full flow. The 
orbit is spiraling on a torus. The black circle shows the corresponding trajectory on the 
Poincar{\'e} return map. Fig.~\ref{fig:3B1} shows the Poincar{\'e} return map 
for the spacecraft for a variety of starting points. 
A variety of orbits are shown, most of which are quasiperiodic invariant circles. 
An exception is A-trajectory in Fig. \ref{fig:3B_global}(a), which is an invariant recurrent set consisting of $42$
circles. Each circle is an invariant quasiperiodic circle under the 42-nd iterate of the map. 
Using our numerical method for 
Birkhoff averages, in the itemized list given below we summarize our results for the quasiperiodic orbit labeled $B_1$. 

The error in the quasiperiodicity computations using weighted Birkhoff averages decreases exponentially 
in the number of iterates $N$ (see Fig.~\ref{fig:3B2}c). This speed of convergence 
means the accuracy of our solutions is the limit of numeric precision. 
In particular, we have computed trajectories for the Poincar{\'e} return map using an 8-th order Runge-Kutta 
method with time step $10^{-5}$, in quadruple precision, meaning our results given 
below are computed up to thirty digits of accuracy. 
Section~\ref{sec:Quasi} formally defines the computed values given in this list. 
\begin{enumerate}
\item The rotation number is given in Table~\ref{table:summary}, computed to 30 digits of accuracy. 
\item We can compute the Fourier series of up to 200 terms. There is a conjugacy map $h$ between the 
first return map and a rigid rotation on the circle. Evaluating the Fourier series allows us to reconstruct the conjugacy map ({\em cf.} Fig.~\ref{fig:3B2}a). 
\item The exponential decay of the coefficients in the Fourier series is a strong indication
of the analyticity of conjugacy function ({\em cf.} Fig.~\ref{fig:3B2}b). 
\end{enumerate}

\begin{figure}[t]
\centering
\subfigure[ ]{\includegraphics[width = .33\textwidth]{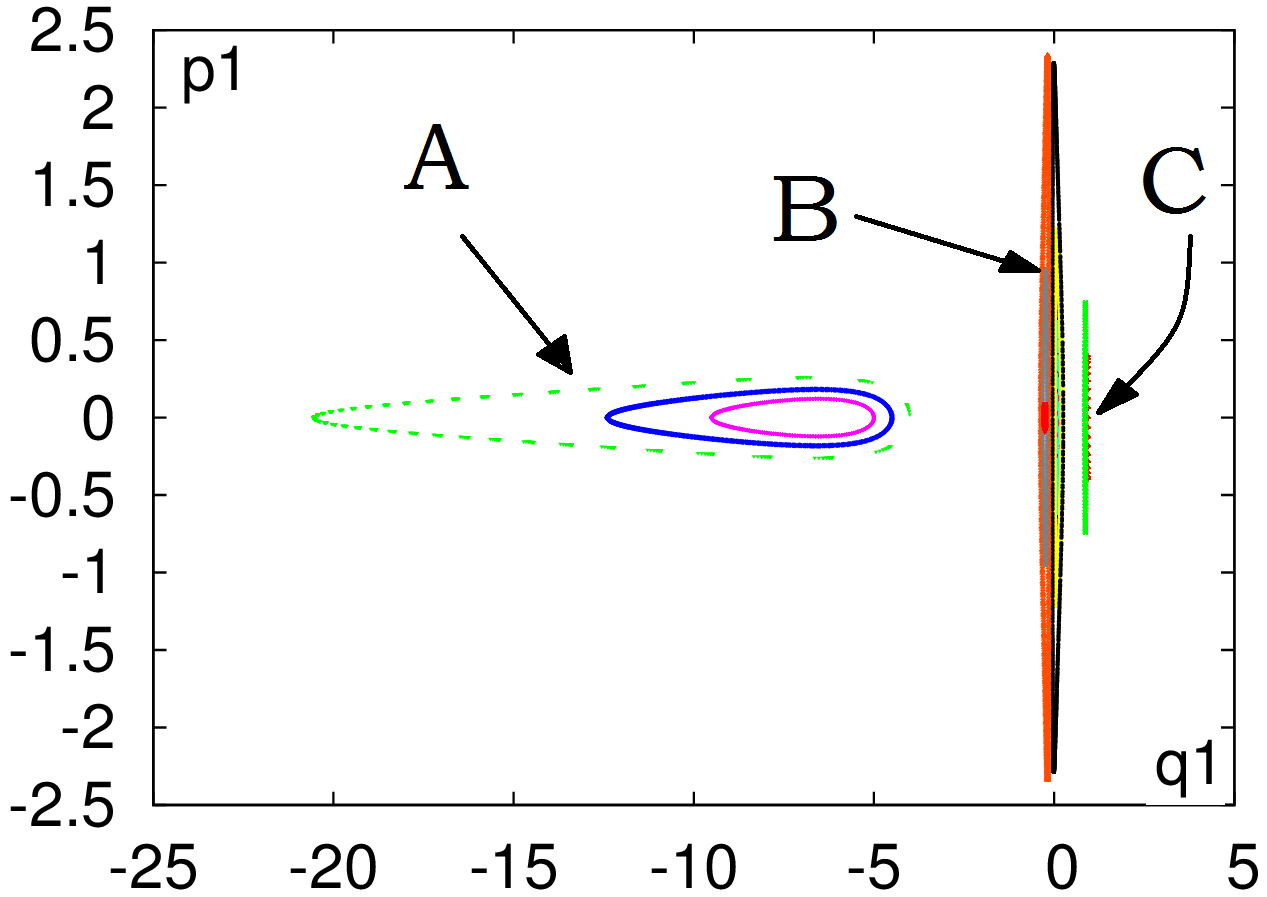}}
\subfigure[ ]{\includegraphics[width = .33\textwidth]{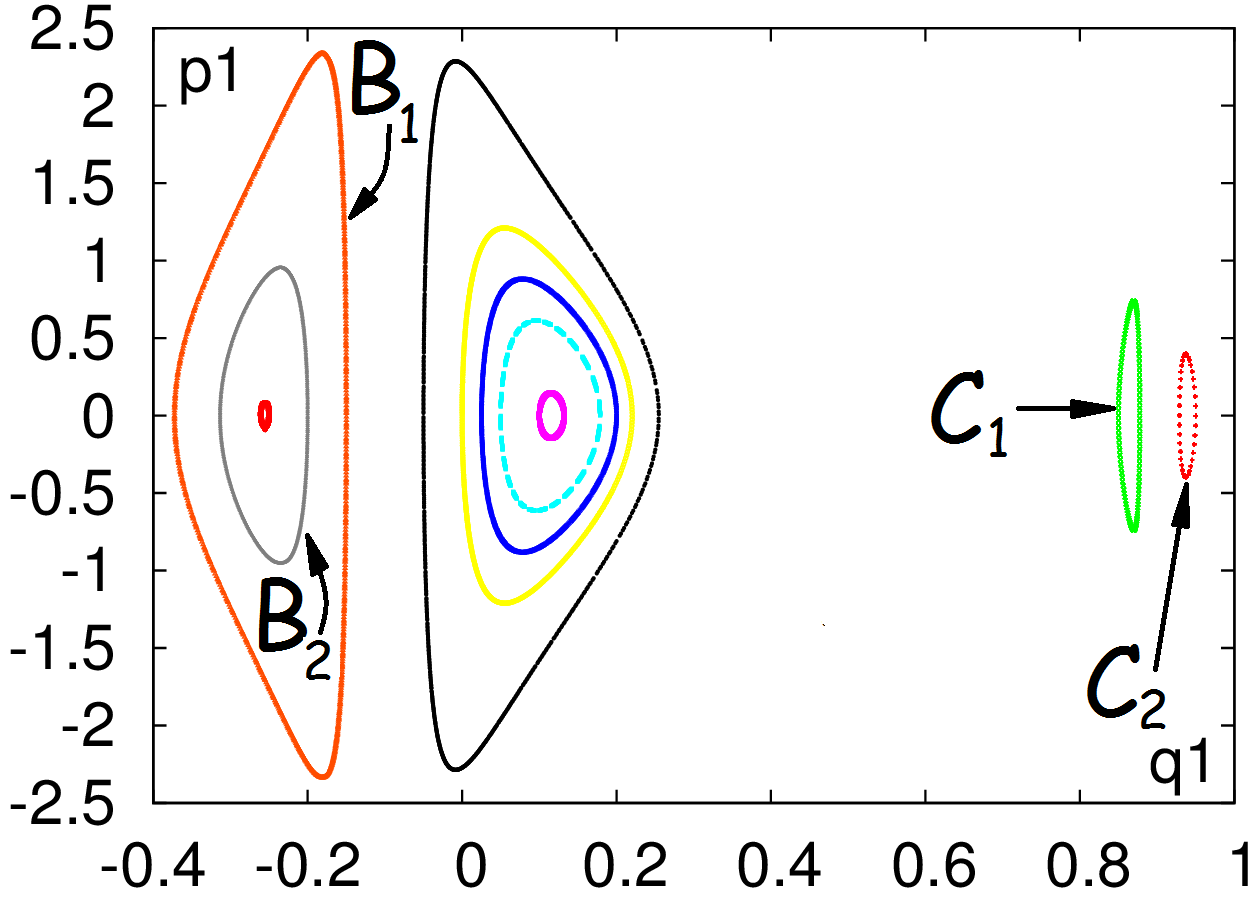}}
\subfigure[ ]{\includegraphics[width = .33\textwidth]{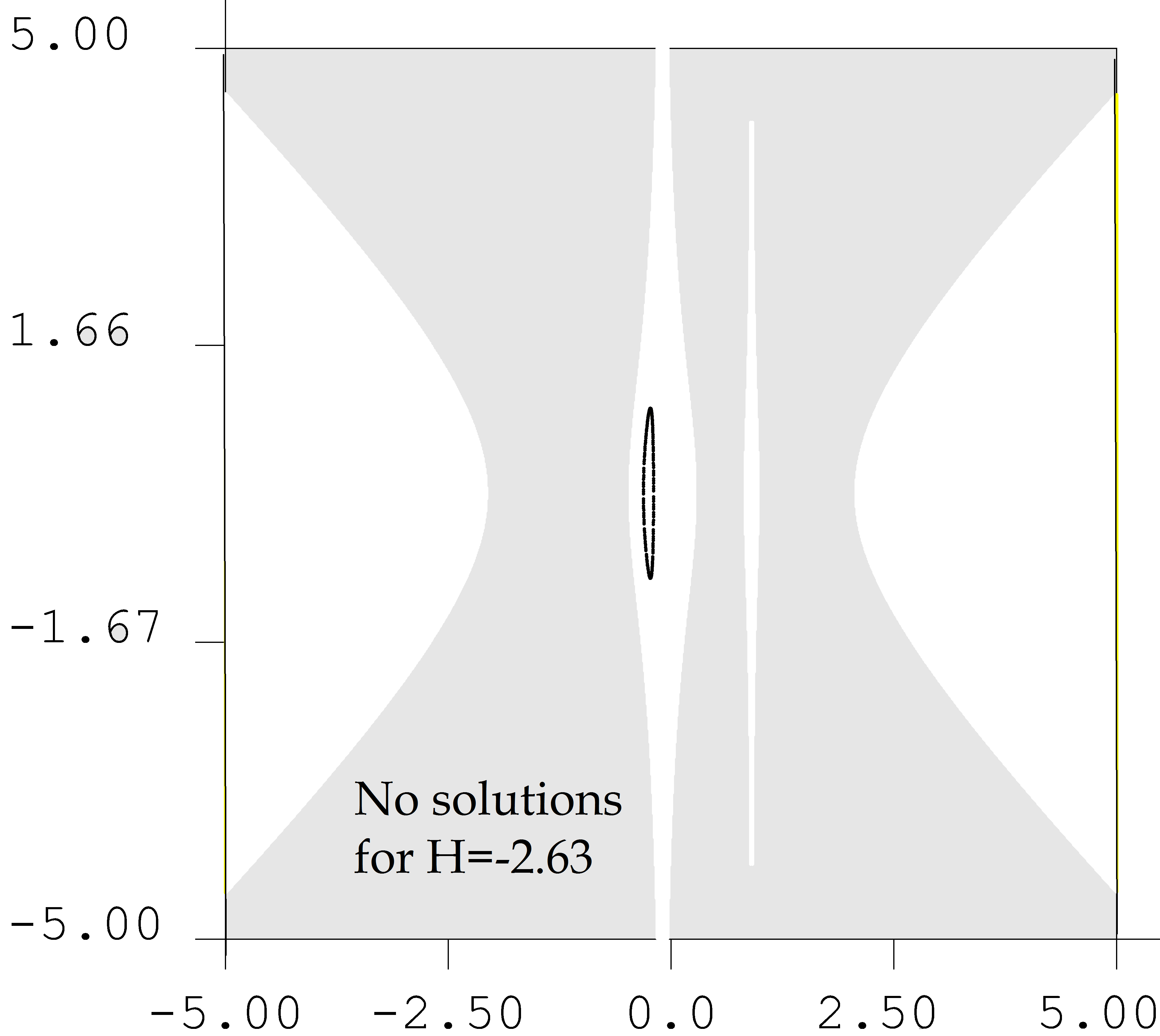}}
\caption{ {\bf Poincar{\'e}-return map for the restricted three-body problem.} \label{fig:3B1}
All three parts show a projection to the $q_1-p_1$ plane. Figures (a) and (b) show various quasiperiodic trajectories 
on the Poincar{\'e} section $q_2=0$, of the restricted three-body problem (\ref{eqn:ThreeBody}). Note that
the planet is fixed at the point $(-0.1,0)$ and the planet at $(0.9,0)$. Thus some trajectories orbit
both the planet-moon system and some only orbit the planet or the moon. Each time the flow hits 
$q_2=0$ and $\frac{dq_2}{dt}>0$, we plot $(q_1,p_1)$. Each trajectory shown is a (topological) circle
with quasiperiodic motion. The energy $H$ for all the circles shown in the figures, including $B_1$ is the same and $H\approx -2.63$. Part (c) shows in white all the initial points $(q_1,p_1)$ on the Poincar{\'e}  surface for which there exists a 
$p_2$ so that the Hamiltonian $H$ at $(q_1,q_2=0,p_1,p_2)$ is the same as the one in parts (a) and (b). Part (c) also shows the trajectory which corresponds to the circle $B_1$ in (a) and (b).}
\label{fig:3B_global}
\end{figure}
\begin{figure}
\centering
\subfigure[ ]{\includegraphics[width = .48\textwidth]{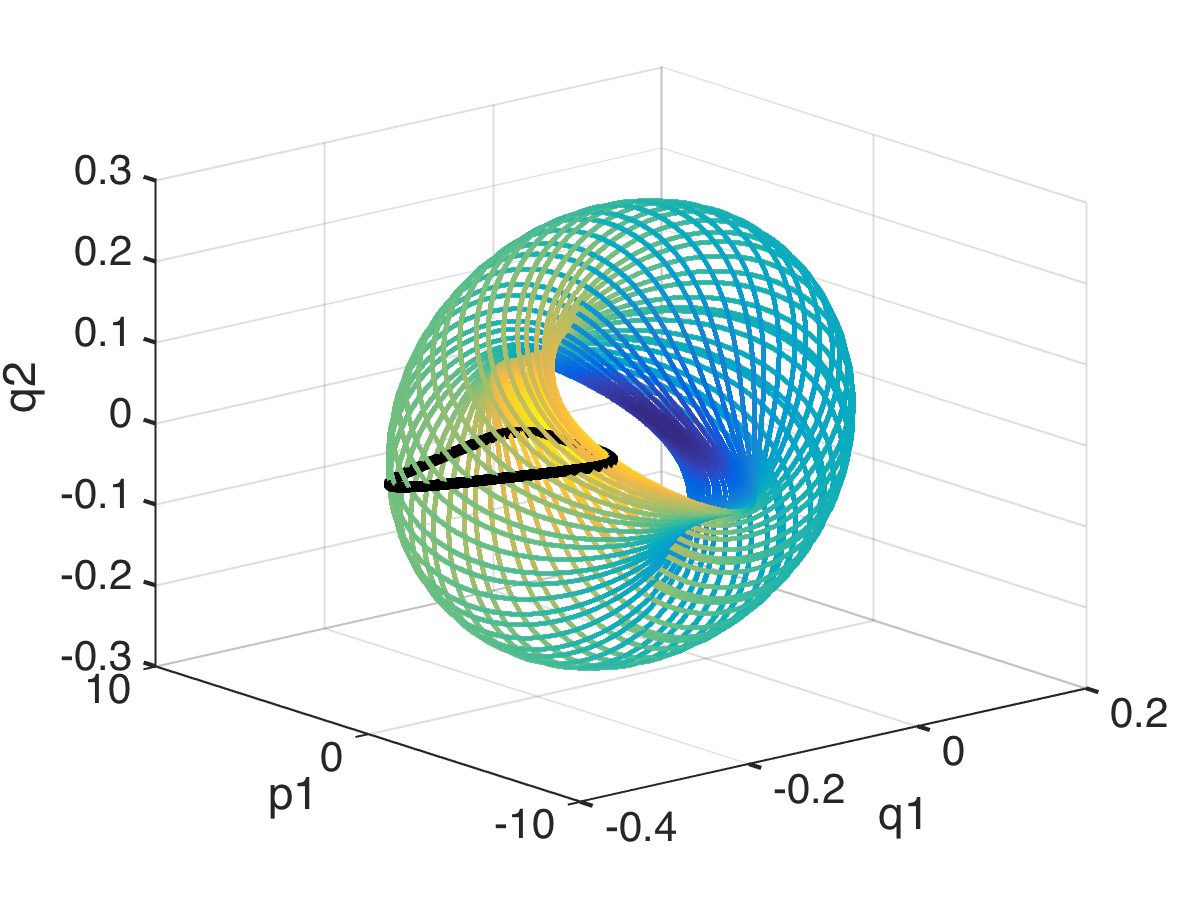}}
\subfigure[ ]{\includegraphics[width = .48\textwidth]{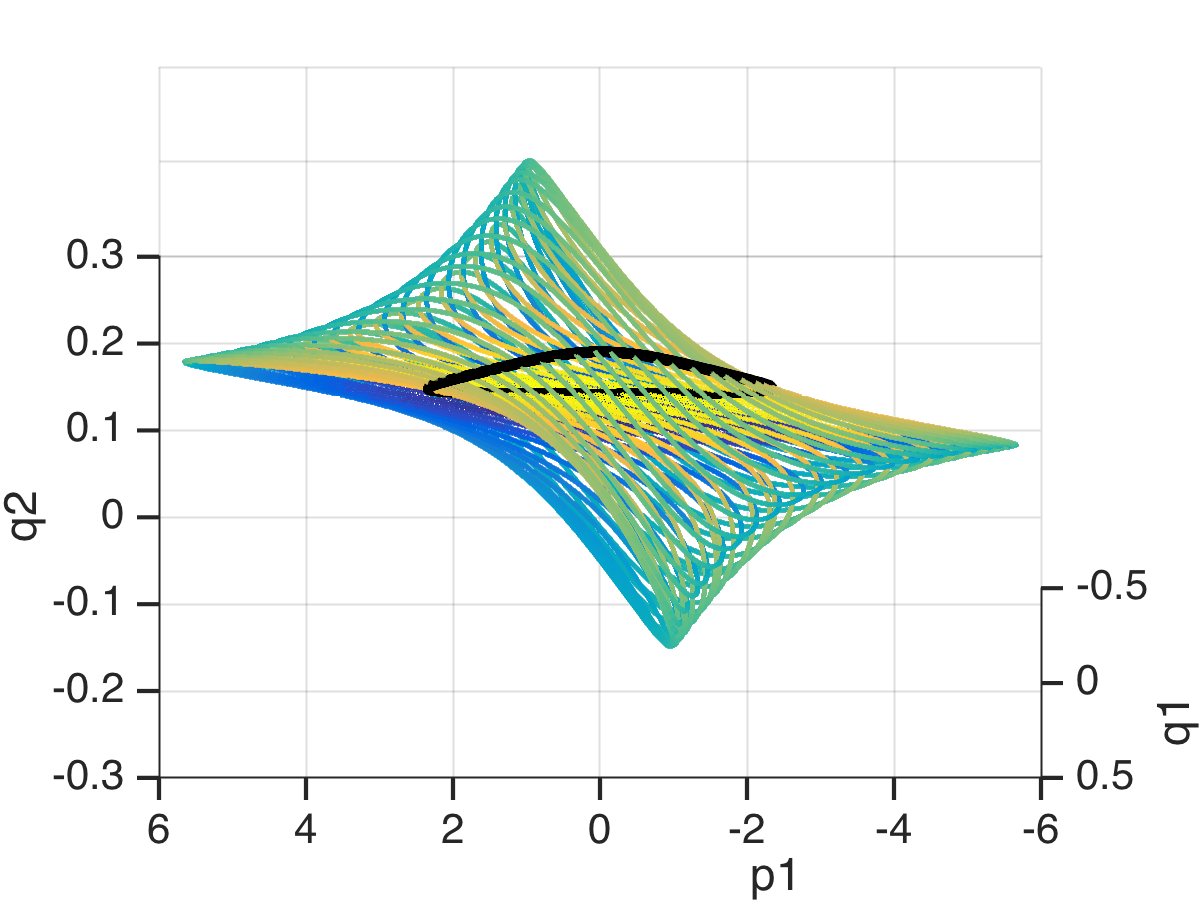}}

\subfigure[ ]{\includegraphics[width = .48\textwidth]{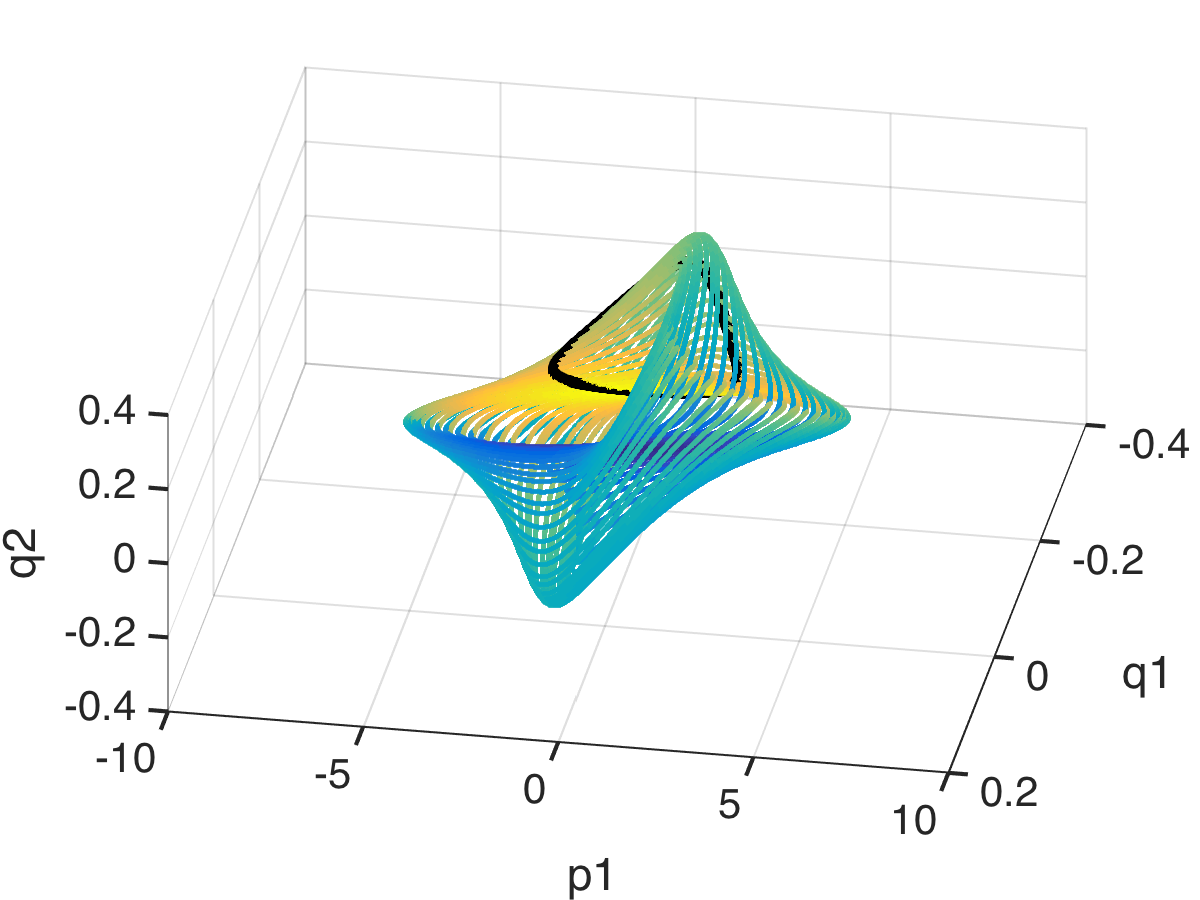}}
\subfigure[ ]{\includegraphics[width = .48\textwidth]{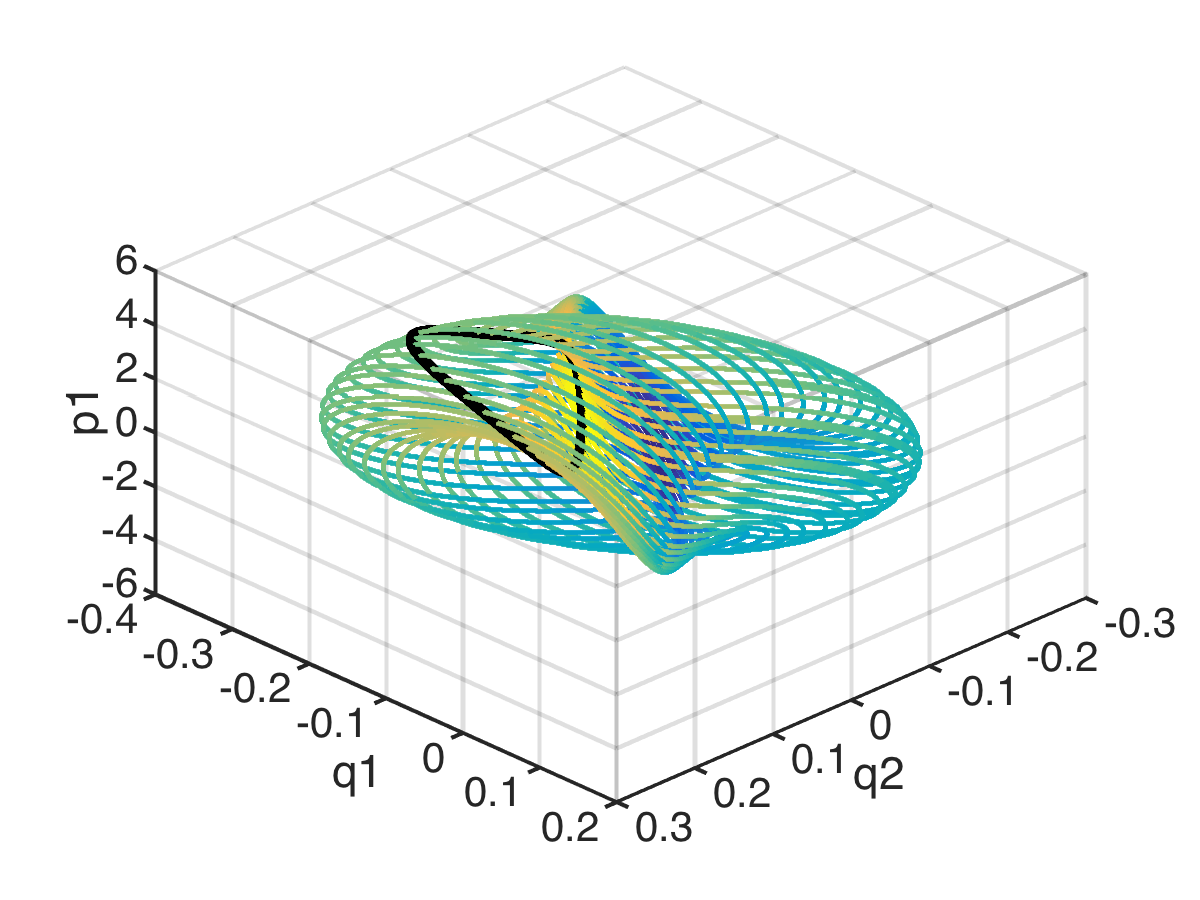}}
\caption{\label{fig:3B3}
{\bf Torus flow for the restricted three-body problem.} 
All four views are of the same two-dimensional quasiperiodic torus lying in $\mathbb{R}^4$. Each picture consists of the same trajectory spiralling densely on this torus. This trajectory is the solution of (\ref{eqn:ThreeBody}), shown as curve $B_1$ in Fig. \ref{fig:3B_global}. We require four different 
views of this torus because the embedding into three dimensions gives a highly non-intuitive images. The black circle is the set of values of the Poincar{\'e} return map with $q_2=0$ for this flow torus.}
\end{figure}
\begin{figure}[t]
\centering
\subfigure[ ]{\includegraphics[width = .4\textwidth]{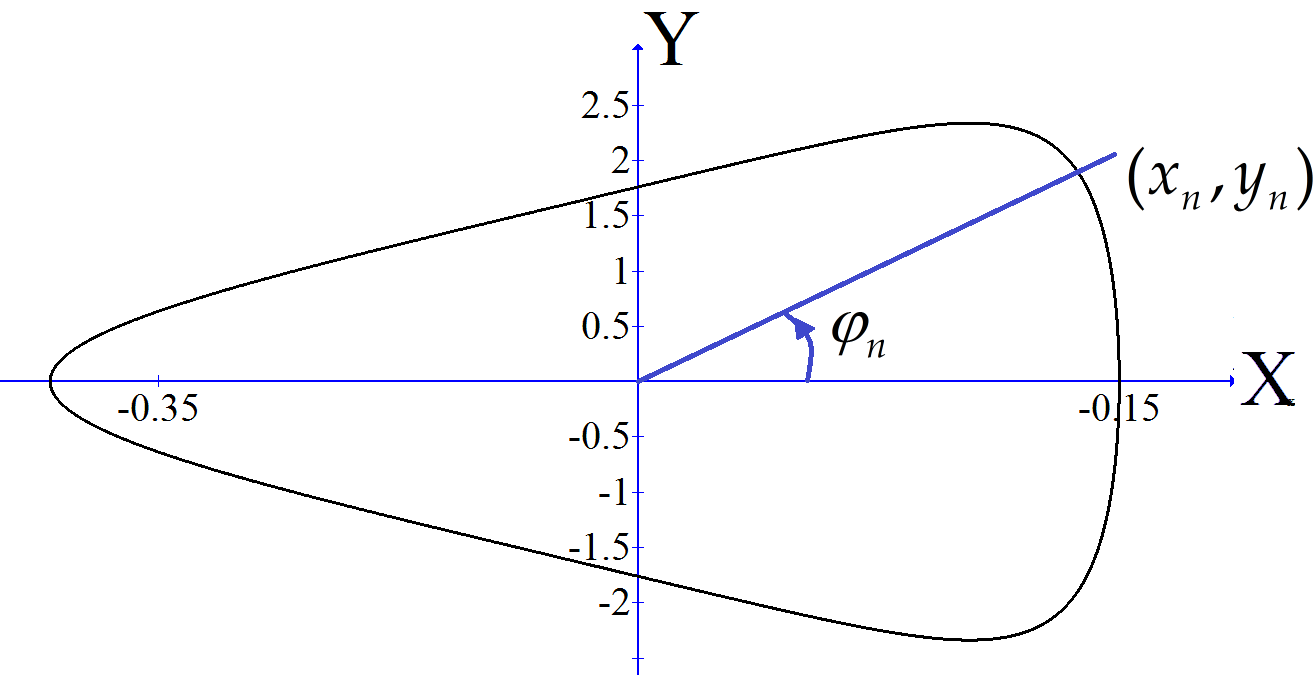}}
\subfigure[ ]{\includegraphics[height=.2\textwidth]{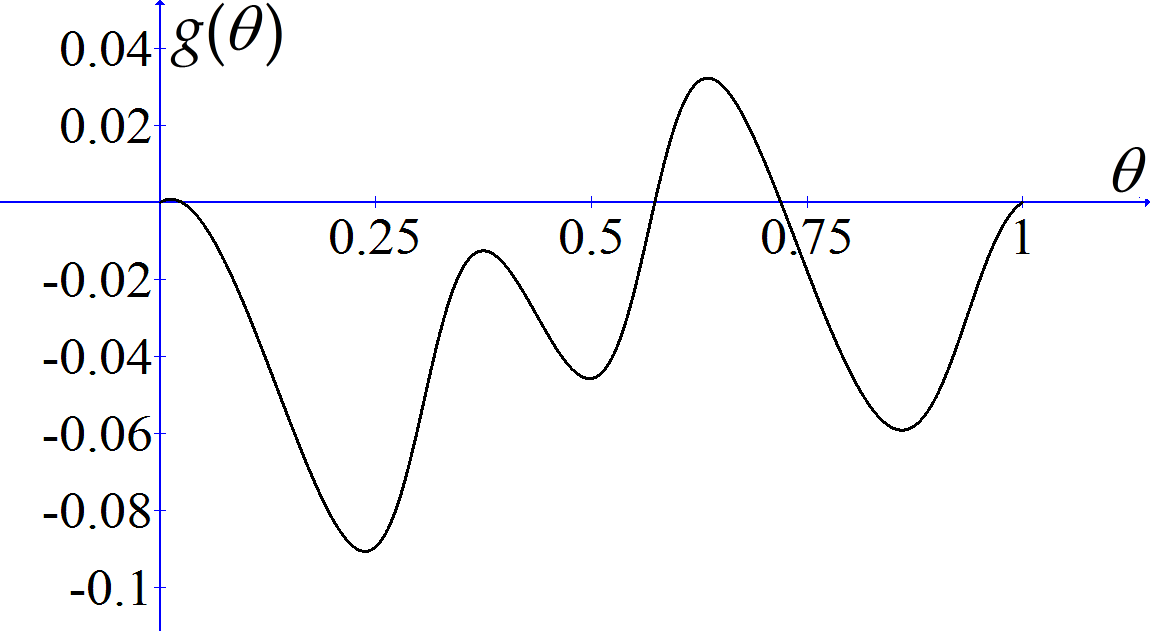}}

\subfigure[ ]{\includegraphics[height= .4\textwidth]{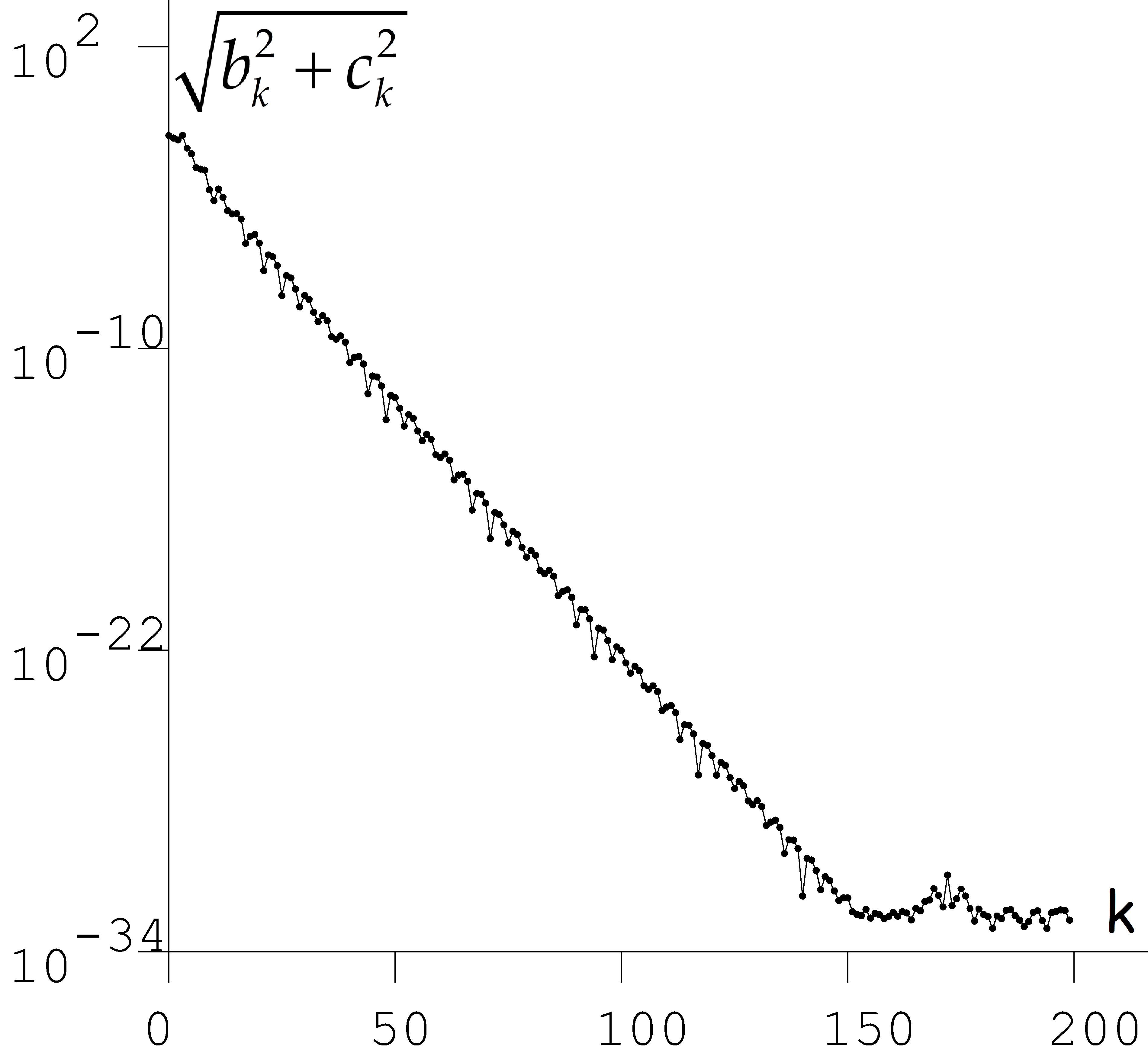}}
\subfigure[ ]{\includegraphics[height= .4\textwidth]{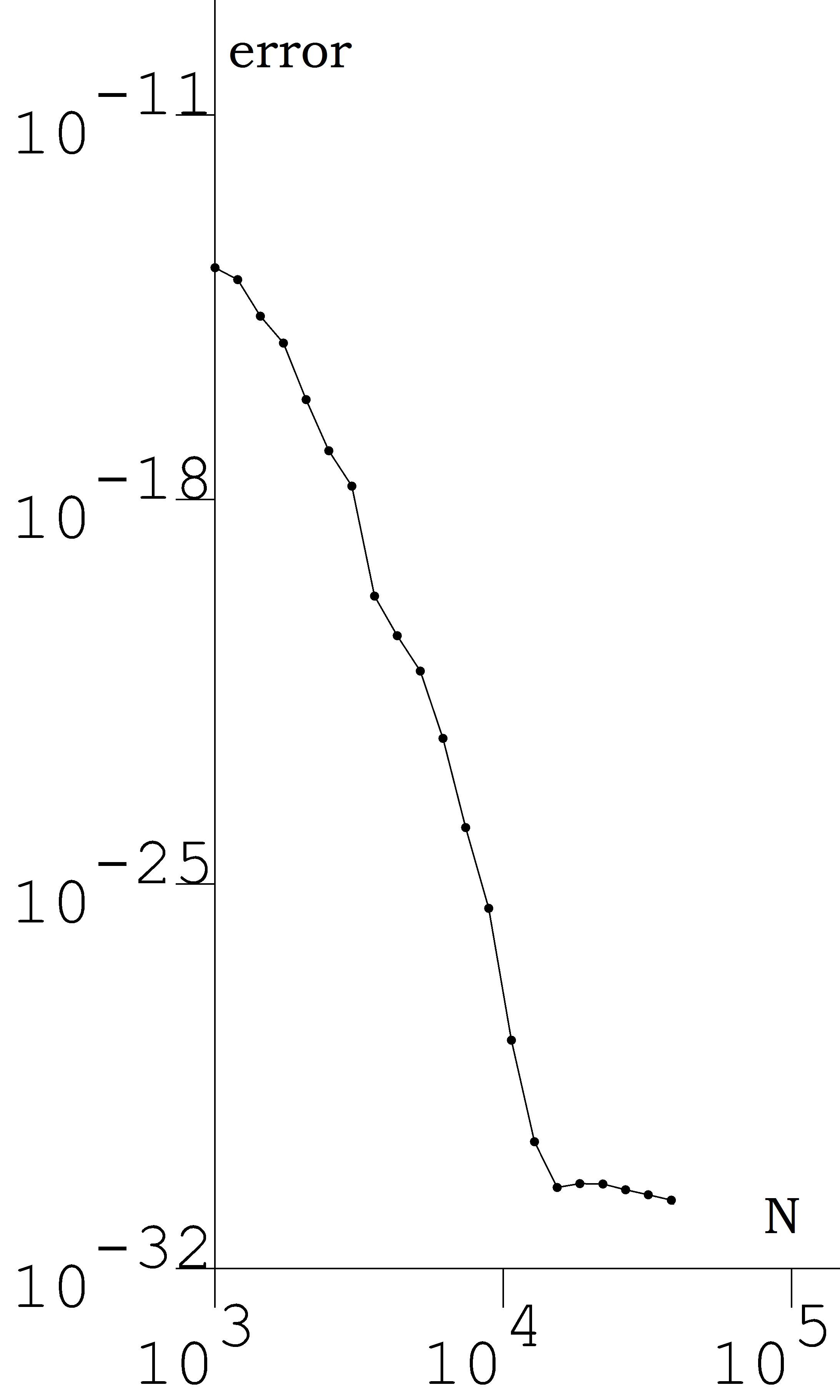}}
\caption{\textbf{Quasiperiodicity for the restricted three-body problem.} For the quasiperiodic circle $B_1$ in Fig. \ref{fig:3B1}, Part(a) shows how the invariant circle is parameterized by coordinates $\phi \in S^1 \equiv [0,1)$. Part (b) depicts the periodic part $g(\theta)$ of the conjugacy between the 
quasiperiodic behavior and rigid rotation by $\rho$. See (\ref{eqn:g_conjufacy}) for a description of $g(\theta)$. Part (c) shows the norm of the Fourier coefficients of the conjugacy as a function of index. This exponential decay 
indicates that the conjugacy function is analytic, up to numerical precision. Part (d) shows the convergence
rate of the error in the rotation number $\rho_N$ as a function of the number of iterates $N$. The ``error'' is the difference $|\rho_N-\rho_{N^*}|$, where $N^*$ is large enough so that $\rho_N$ appears to have converged.
The step size used for the Runge-Kutta (RK8) scheme is $10^{-5}$.
}
\label{fig:results_3B}
\label{fig:3B2}
\end{figure}
\begin{table}[t]
\begin{center}
{\begin{tabular}{ |c | c| c | c | }
\hline
Example & Equation & Rotation number(s) & Related Figures \\
\hline
Restricted three-body problem & \ref{eqn:ThreeBody} & $0.063961728757453097164077919081024$ & \ref{fig:3B1}, \ref{fig:3B2}, \ref{fig:3B3} \\[1pt]
Standard map & \ref{eqn:StdMap} & $0.12055272197375513300298164369839$ & \ref{fig:StdMap_global}, \ref{fig:StdMap} \\[1pt]
Forced van der Pol oscillator, $F=5$ & \ref{vdPol_forced} & $0.29206126329199589285577578718959$ & \ref{fig:vdP_global} \\[1pt]
Forced van der Pol oscillator, $F=15$ & \ref{vdPol_forced} & $0.37553441113144010884908928083318$ & \ref{fig:vdP_global} \\[1pt]
Forced van der Pol oscillator, $F=25$ & \ref{vdPol_forced} & $0.56235370092685056634419221336154$ & \ref{fig:vdP_global} \\[1pt] 
Two-dimensional torus & \ref{eqn:2D_map} & $\rho_1, \rho_2$ in Table~\ref{table} & \ref{fig:2DMap_overview}, \ref{2DTorus_results} \\[1pt]
\hline
\end{tabular}}
\end{center}
\caption{\label{table:summary}
{\bf Summary of our numerical calculations.} 
}
\end{table}

\clearpage
\section{Quasiperiodicity}\label{sec:Quasi}

In the introduction, we described quasiperiodic motion as motion that
could be fully understood through a set of angles of rotation. We now
formalize that idea in the following definition for
quasiperiodicity.

\begin{definition}[Quasiperiodicity]
For a dimension $d \ge 1$, let $\rho = (\rho_1,\rho_2,\dots, \rho_d)$ be a rotation vector
such that all $\rho_k$ are irrational and irrationally related (defined below in (\ref{eqn:irratrel})).
Then following map is known as a {\bf rigid irrational rotation}:
\begin{equation}\label{eqn:Pure_rotation}
T_{\vec\rho}(\theta) = \theta + \rho \pmod 1, \mbox{ where } 0 \le \theta_k \le 1, k=1,2,\dots,n.
\end{equation}
A rigid rotation the simplest, albeit least interesting example of a map
with quasiperiodicity. Since $T_{\vec\rho}$ gives the same values on opposite sides
of the unit cube, we identify the sides and refer to the domain of the
rigid rotation as a {\bf circle} in one dimension and an {\bf $n$-torus}
in $n>1$ dimension. From now on we will refer to the circle as
a 1-torus. We define a general map $T$ to be {\bf
quasiperiodic} if either $T$ or some iterate $T^k$ is topologically
conjugate to a rigid rotation. (We will assume $k=1$ in the rest of this
description.) That is, a map $T$ is quasiperiodic if there is a rigid
rotation map $T_{\vec\rho}$ and an invertible {\bf conjugacy map} $h$ such that
\[T(h(\theta)) = h(T_{\vec\rho}(\theta)).\] A flow has quasiperiodic behavior if
its associated first return Poincar{\'e} map has quasiperiodic behavior.

Since we are focused on maps, we restate this definition in
terms of iterates. Let $(x_n)$ be the forward orbit under $T$, and $( \theta_n )$ the
forward orbit under $T_{\vec\rho}$. That is, $x_{n+1} = T(x_n)$ and $\theta_{n+1} = \theta_n + \rho \pmod 1$. Then as long as $\phi_0 = h(\theta_0)=h(0)$,
\begin{equation}\label{eqn:conjugacy1D}
\phi_n = h(\theta_n)~\mbox{ for all } n=0,1,2,3,\ldots .
\end{equation}
The map $h$ maps a torus to the domain of $T$, meaning that the domain of $T$ is
a topological circle or torus. Since $h$ is a homeomorphism, the map
\begin{equation}\label{eqn:g_conjufacy}
g(\theta) := h(\theta)-\theta
\end{equation}
is periodic.
\end{definition}

Thus a map $T$ restricted to an invariant topological circle $C$ is
quasiperiodic if there is an continuous invertible map $h$ from $C$ to the true circle
which maps $T$ to a rigid irrational rotation map. 

For an invertible map $F$ to be quasiperiodic on a topological circle
$C$, it is necessary and sufficient that $F$ has a dense trajectory, as
shown in \cite{TransCircleHomeo}. In general, a circle homeomorphism
without periodic points may not be quasiperiodic. However, if we assume
that the map $F$ and the topological circle $C$ are twice continuously
differentiable, then Denjoy~\cite{VanKampen} showed that these
conditions are both necessary and sufficient. Furthermore, clearly any
rigid irrational rotation map is a real analytic map, but even if we
assume that a quasiperiodic function is analytic, Arnold showed that the conjugacy map $h$
may only be continuous for some atypical rotation number. However, Herman (see \cite{HermanSeminal}) proved that for circle
homeomorphisms, most conjugacy maps $h$ are analytic.

\clearpage
\section{Our $\Q_N+$ method and its applications}
\subsection{Rotation number}\label{sec:Method_Rot}

Here is our method for calculating the limit of the Birkhoff average $\lim_{N \to
\infty} \Sigma_{n=1}^N f(x_n)/N = \int f \; d\mu$ 
along an ergodic trajectory (or first return map) $(x_n)$. 
Let the weighting function $w$ be defined as 
\begin{equation}\label{eqn:weight}
w_{exp}(t) :=w(t) :=\begin{cases}
\exp\left(\frac{1}{t(t-1)}\right), & \mbox{for } t\in(0,1)\\
0, & \mbox{for } t\notin(0,1).
\end{cases}
\end{equation} 
Let $T$ be quasiperiodic on some set $X_0$, with $X_0$ a topological torus. Henceforth, $X_0\equiv\torus$ will denote a $d$-dimensional topological torus. For $d=1$, $\torus$ is a topological circle. For a continuous function $f$ and a $C^\infty$ quasiperiodic map $T$ 
on $\torus$, let $x_n\in\torus$ be such that $x_n = T(x_{n-1})$ for all $n>1$. We define a {\bf Weighted Birkhoff ($\Q_N$) average} 
of $f$ as 
\begin{equation}\label{eqn:QN}
\Q_N(f)(x) :=\sum_{n=0}^{N-1} \hat{w}_{n,N}f(x_n),\mbox{ where }\hat{w}_{n,N}=\frac{w(n/N)}{\Sigma_{j=0}^{N-1}w(j/N)}
\end{equation}

In a companion paper~\cite{Das-Yorke}, it is shown that if $f$ and $T$ are $C^\infty$, 
then for every positive integer $m$ there is a constant $C_m >0$ such that 
\[\left| WB_N(f)(x) - \int f d\mu \right| \le C_m N^{-m} \mbox{ for every }m.
\]
We refer to the above as {\bf super} (polynomial) {\bf convergence} or {\bf exponential convergence}.

Our numerical method converges fast enough to allow us to obtain high precision values for $\int f d\mu$ 
with relatively low computational cost. In particular, by appropriate choices for $f$, 
we obtain rotation numbers and sometimes the Lyapunov exponents,
Fourier coefficients, and conjugacy reconstructions for quasiperiodic
maps and flows potentially in any finite dimension, though here we work mainly in dimension 
one, with one example in dimension two. We now show
how to apply this general method to computation of specific quantities. We observe that $N$ must be larger for $\mathbb{T}^2$ than for $\mathbb{T}^1$, to get a 30-digit accuracy.

{\bf Calculation of rotation vectors.}
\begin{figure}[t]
\centering
\subfigure[ ]{\includegraphics[width = .35\textwidth]{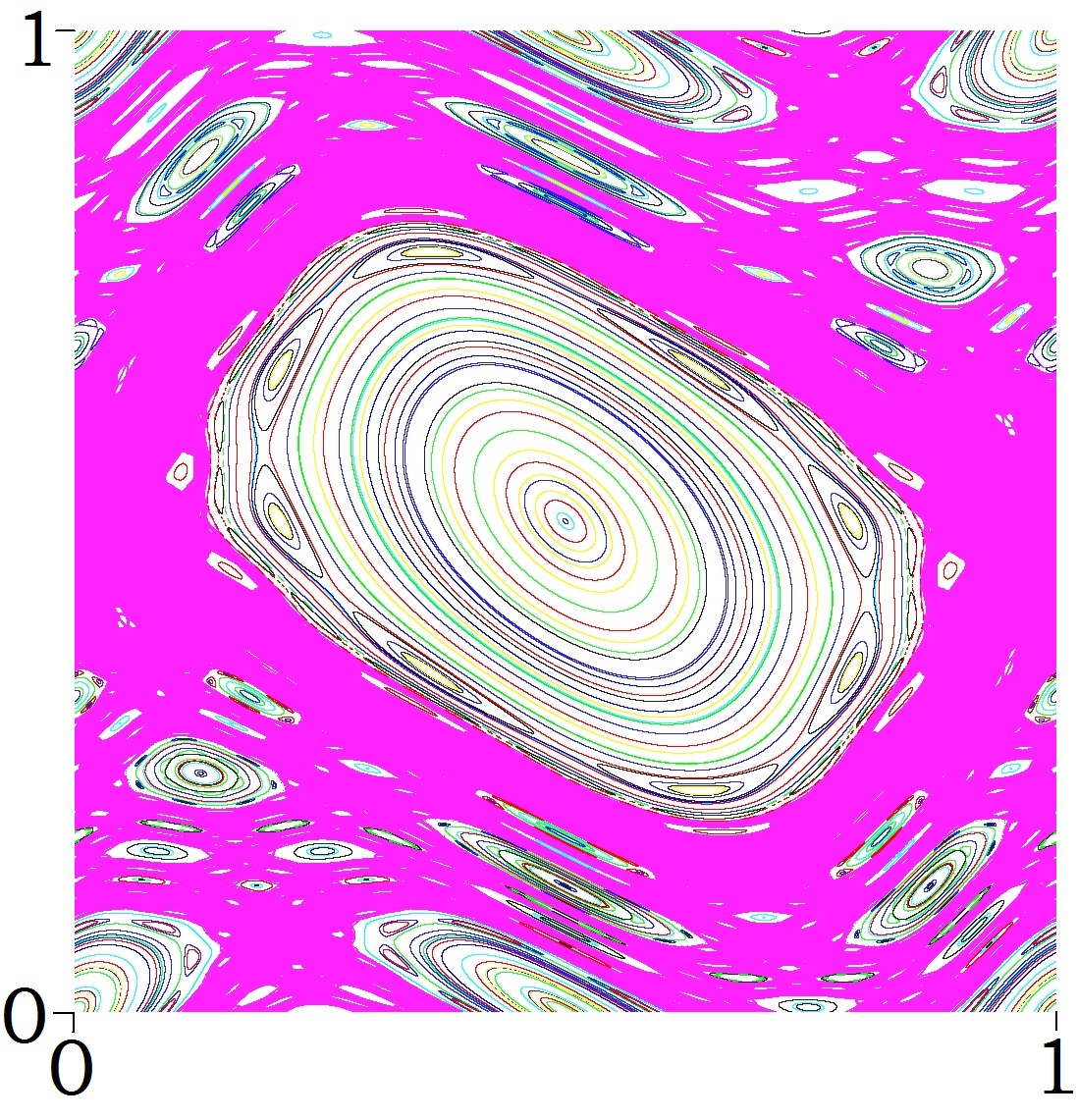}}
\subfigure[ ]{\includegraphics[width = .35\textwidth]{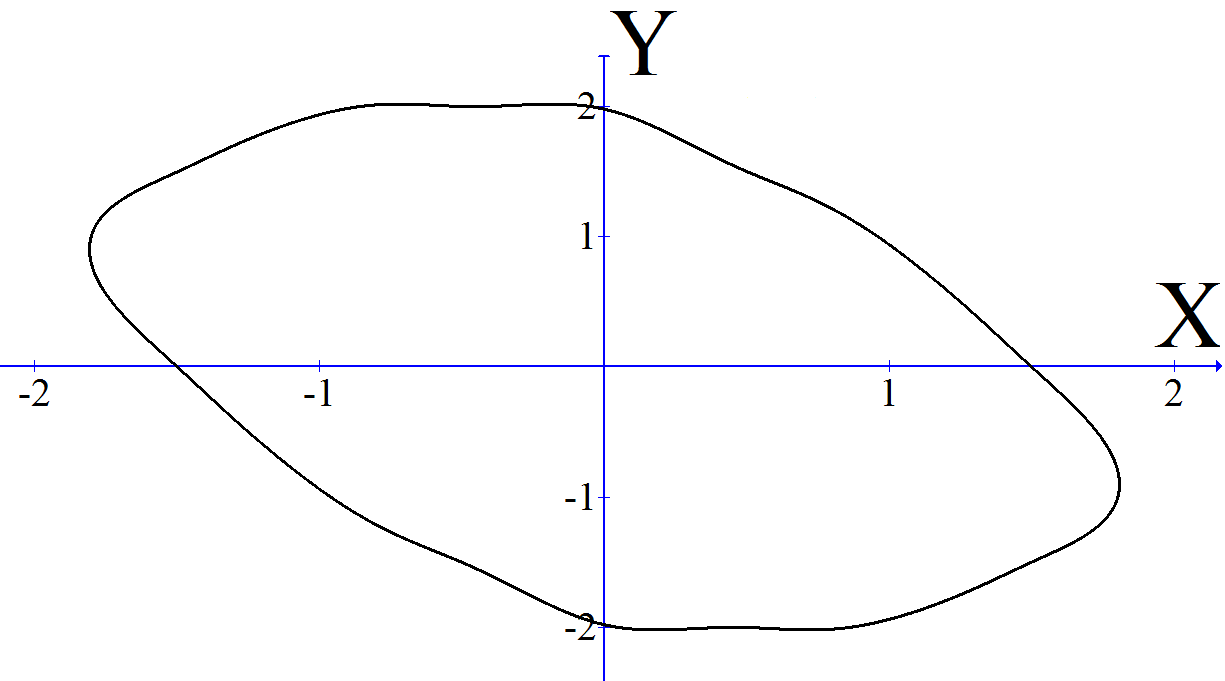}}
\caption{
{\bf The standard map.} A variety of orbits from different initial
conditions in the standard map $S_1$ defined in~(\ref{eqn:StdMap}) are plotted on the
left. We can see both chaos (in pink) and quasiperiodic orbits under this map. A single
topological circle with quasiperiodic behavior is plotted on the right.
The orbit has initial conditions 
$(x, y) \approx (-0.607, 2.01)$.
That is, if we restrict the map to this invariant circle, then it is
topologically conjugate to a rigid irrational rotation.
}
\label{fig:StdMap_global}
\end{figure}
\begin{figure}[t]
\centering
\subfigure[ ]{\includegraphics[width = .35\textwidth]{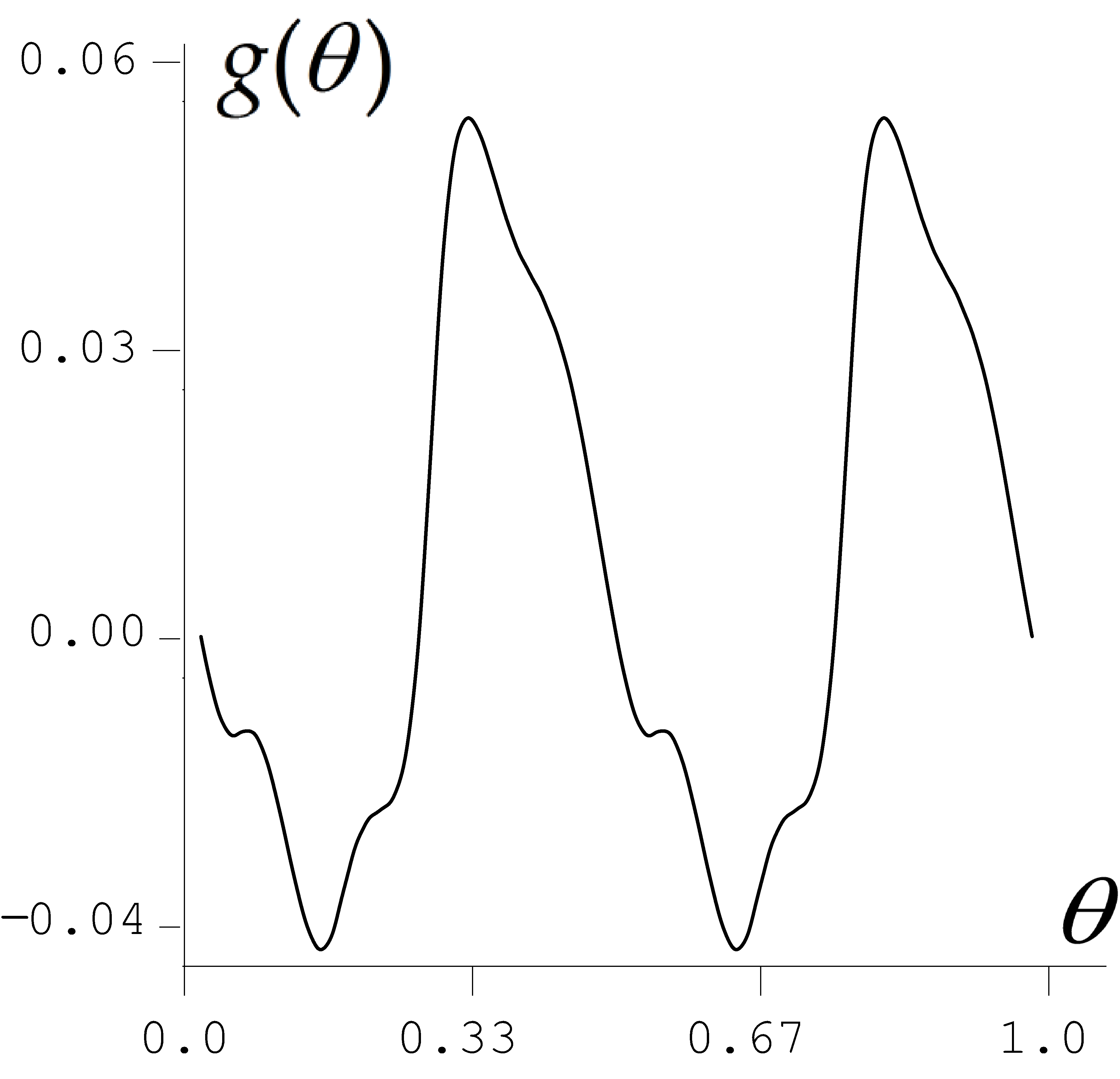}}
\subfigure[ ]{\includegraphics[width = .35\textwidth]{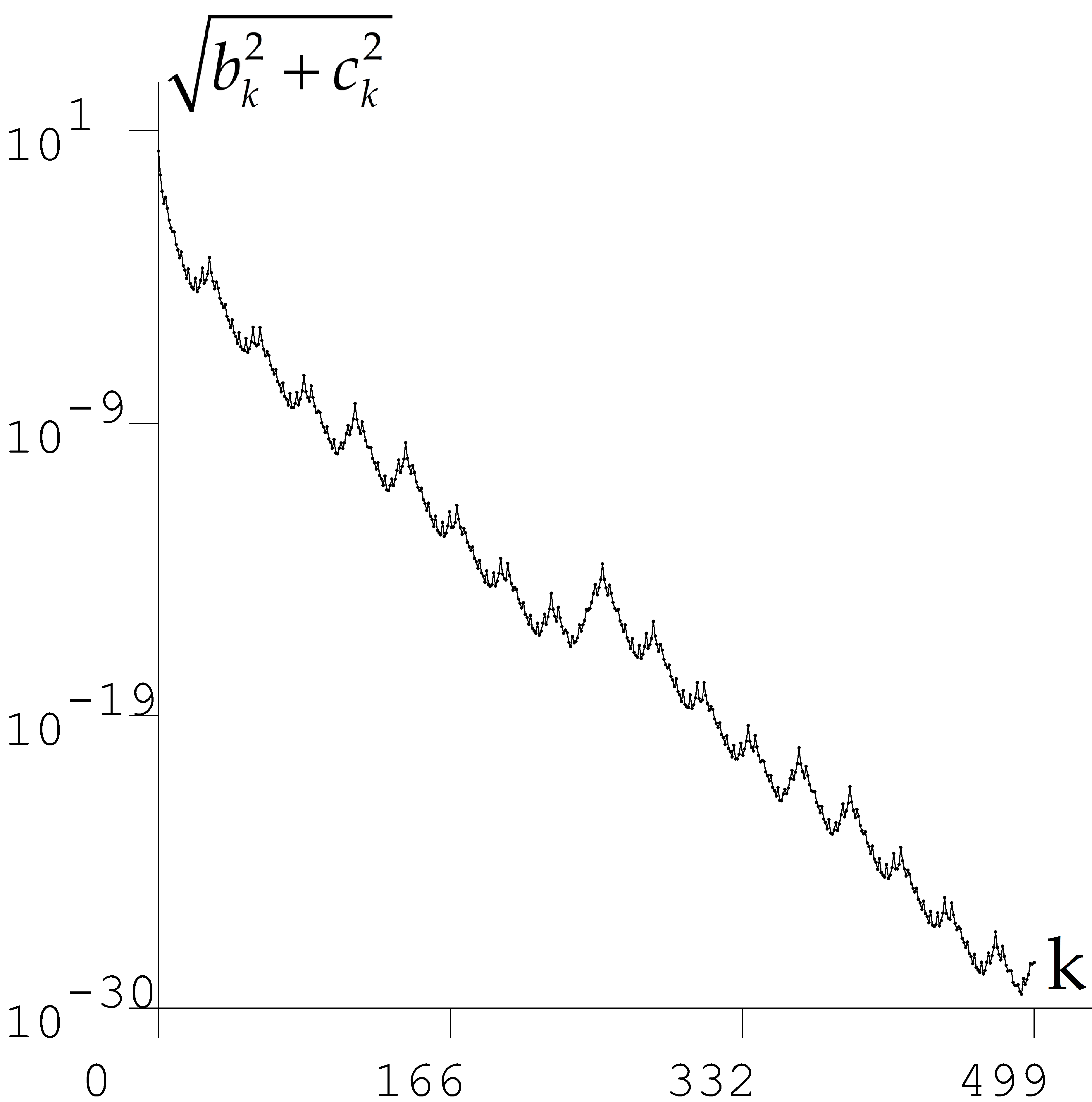}}
\caption{\textbf{The Standard map conjugacy.} 
Fig. (a) shows the shows the
change of variables which converts the motion on the circle in Fig.~\ref{fig:StdMap_global}b into a pure
rotation. 
Fig. (b) shows the decay of the Fourier coefficients. Since
the conjugacy is an odd function, the odd-numbered Fourier sine and
cosine terms are zero and therefore have been omitted from the picture.
The decay of the Fourier terms can be bounded from above be an
exponential decay, which suggests that the conjugacy is analytic.}
\label{fig:StdMap}
\end{figure}
\begin{figure}[t]
\centering
\subfigure[ ]{\includegraphics[width=.3\textwidth]{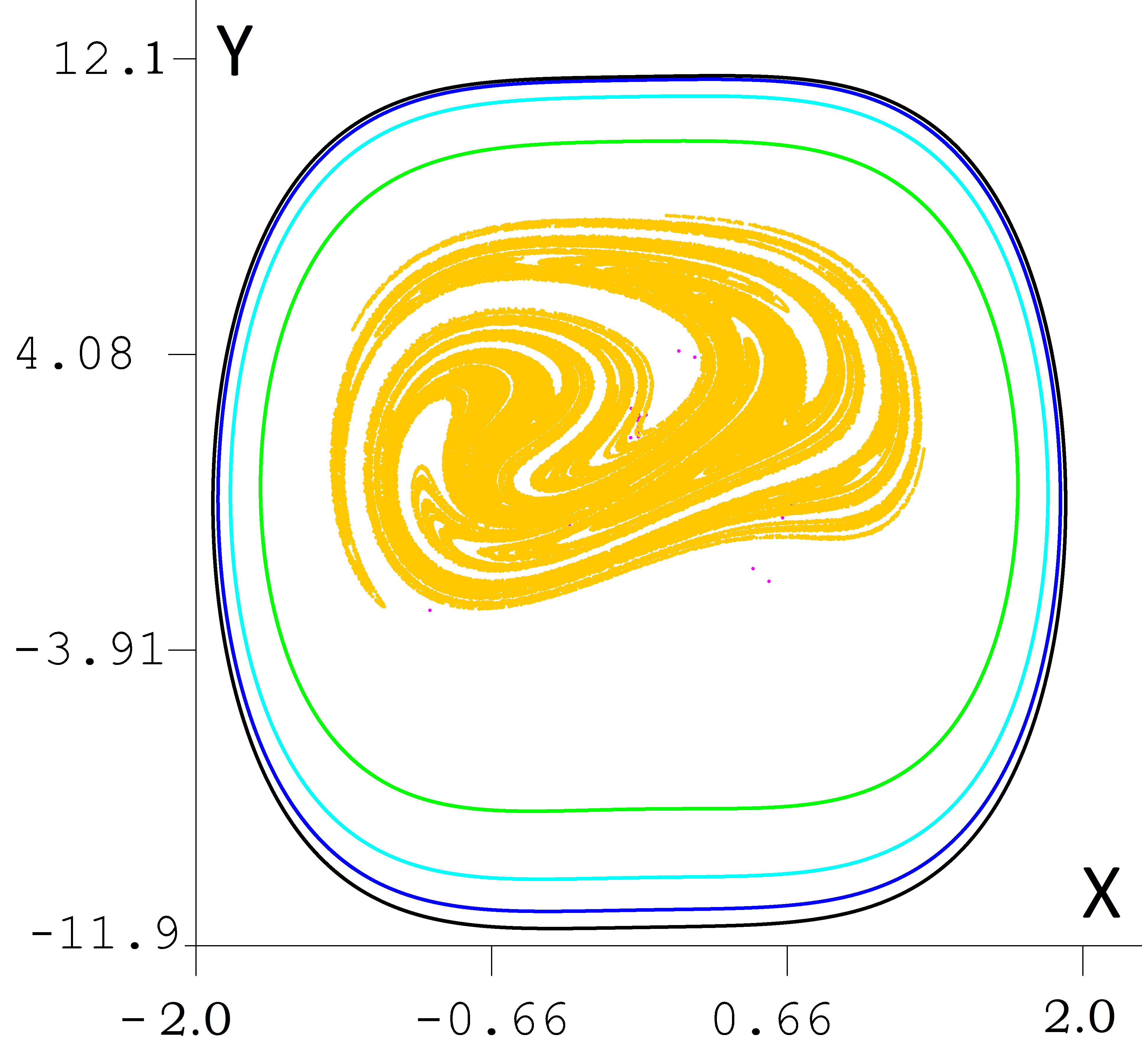}}
\subfigure[ ]{\includegraphics[width = .3\textwidth]{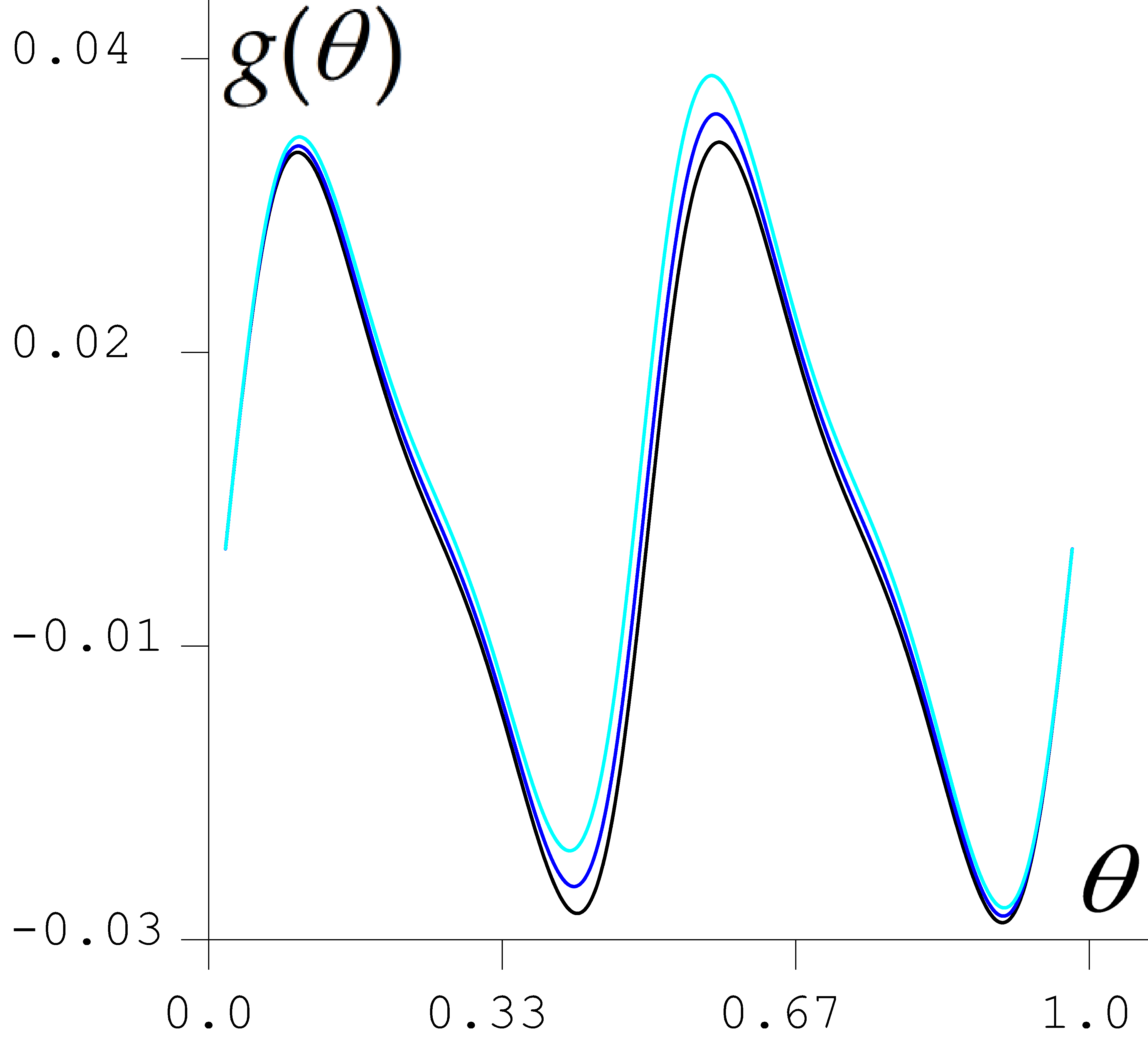}}
\subfigure[ ]{\includegraphics[width = .3\textwidth]{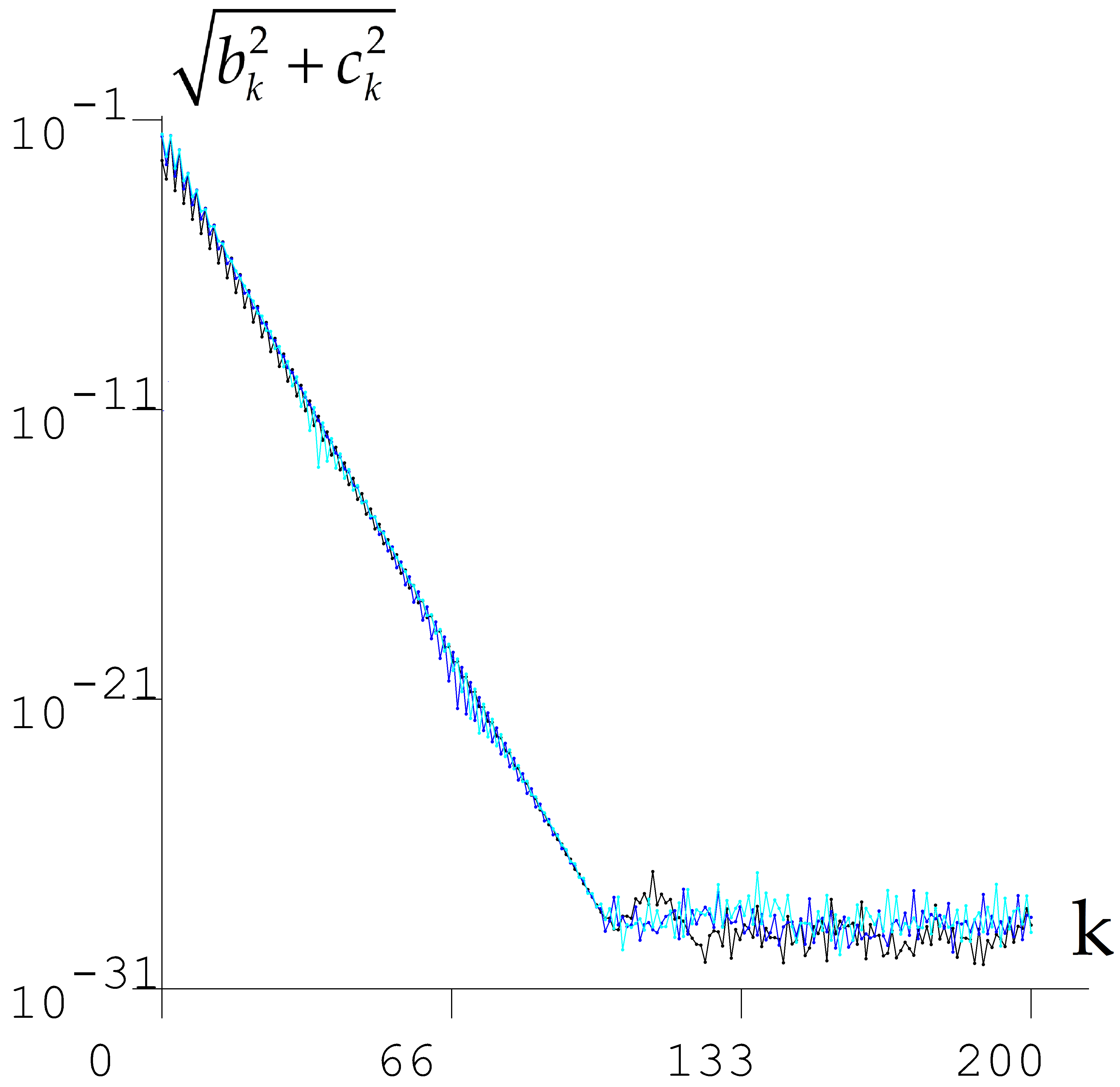}}
\caption{ {\bf Forced van der Pol oscillator.} \label{fig:vdP_global}
Part (a): Attracting orbits for a number of different forcing values $F$ for the stroboscopic
map of the van der Pol flow given in (\ref{vdPol_forced}). The plot
depicts points
$(X,Y) = (x(t_k),x'(t_k))$, where $t_k = 2 k \pi/0.83, k = 0,1,2,\dots.$
The chaotic orbit lying inside the cycles corresponds to $F=45.0$. There are stable quasiperiodic orbits shown as circles, which from outermost to innermost correspond to $F=5.0$, $15.0$, $25.0$ and $35.0$ respectively. Part (b): the periodic part $g(\theta)$ of
the conjugacy (Eq. \ref{eqn:conjugacy1D}) to a pure rotation, for $F=$ 5, 15 and 25. Part (c): the second figure gives a demonstration of the analyticity of the conjugacy to a pure rotation by studying the decay of the norm of the $k$-th
Fourier coefficients with $k$, up to the resolution of the numerics.
}
\end{figure}
Recall that every quasiperiodic map is conjugate to a pure rotation $T_\rho$ by a vector $\rho = (\rho_1, \dots,\rho_n)$. In our discussion, we assume that the components 
of $\rho$ are irrational and irrationally related. 
When we
say that $\rho_1,\dots,\rho_d$ are irrationally related, we mean that
there are no integers $k_1,\dots,k_d $
and $k_0$ for which,
\begin{equation}\label{eqn:irratrel}
k_1\rho_1 + k_2 \rho_2 + \dots + k_d \rho_d \ne k_0
\end{equation}
unless all $k_i$ are zero.
Since we are only provided with the
map $T$ and not with the map $R$, we need a way of recovering this
rotation vector using only a forward trajectory of the map $T$. The
following formula gives a way of calculating the rotation vector $\rho$
directly from the original map. The rotation vector given in the formula
is valid well beyond the case of a quasiperiodic function. For example, the rotation vector can
certainly be rational in one or all of its components. 
%
\begin{definition}[Rotation vectors.]
Let $q(x) = x \pmod 1$ and define $q:\mathbb{R}^d \to \torus$ to be the projection map modulo
$\mathbb{Z}$ in each coordinate. For example, in one dimension the circle $[0,1]$ with ends identified is the image of the real
$\mathbb{R}$ under the map $q$.
For any homeomorphism $T:\torus \to \torus$, let $\bar{T}:\mathbb{R}^d\rightarrow\mathbb{R}^d$ be the lift of
$T$. That is,
\begin{displaymath}\label{eqn:q_as_cover}
q\circ\bar{T}=\phi\circ q.
\end{displaymath}
Since $T$ is a homeomorphism, the map $\bar{T}(x)-x$ must be periodic in $x$ with period 1.
The rotation vector of a homeomorphism $T: \torus \to \torus$ starting at lift point $z$ is given by
\begin{equation}\label{eqn:rot_num_defn}
\rho(T,z):=\underset{n\rightarrow\infty}{\lim}\frac{\bar{T}^n(z)-z}{n}.
\end{equation}
\end{definition}
Note that possible rotation numbers $(\rho_1,\rho_2)$ of a quasiperiodic map on a torus $\torus$ depend on the coordinates chosen for $T$ and $z$. For example,
for domain $\mathbb{T}^2$, the set of pairs $(\rho_1,\rho_2)$ possible for a fixed $T$ is dense in 
the two-torus. See~\cite{Das-Yorke}. 

Based on the above definition, the standard method of computing $\rho$ from an orbit $(y_k)$
of length $n$ is to calculate the average
\begin{equation}\label{eqn:rot_num_alt}
\rho\approx\frac{1}{N} \sum_{n=1}^{N} (\bar{T}(y_n)-y_n).
\end{equation}
This method converges slowly at best, with order of only
$1/N$~\cite{seara:villanueva:06}. 
However, since Eq. \ref{eqn:rot_num_alt} can be written as a Birkhoff average by writing $f(y_n)=T(y_n)-y_{n}$, we can apply our method 
to this function. That is, let $(y_k)_{k=0}^n$ be a forward orbit for $\bar{T}$.
Our approximation of $\rho$ is given by the weighted average of $f$,
\begin{equation}\label{eqn:weighted_rho}
\WB_N (y_{n+1}-y_n) := \sum\limits_{n=0}^{N-1} \hat{w}_{n,N} (y_{n+1}-y_n)\rightarrow \rho.
\end{equation}

{\bf Rotation number for the restricted three-body problem.} In the
introduction, we stated our results, including the rotation number for a
quasiperiodic orbit for the restricted three-body problem.
Fig.~\ref{fig:3B2}(b) shows the convergence rate of the calculation for
this rotation number. .

{\bf Rotation number for the standard map.} 
The standard map is an area preserving map on the two-dimensional torus, often studied as a typical
example of analytic twist maps (see \cite{StdMap1}). It is defined as follows\footnote{The 
standard map generally depends on a parameter $\alpha$, but we only
consider the case $\alpha=1$.}
\begin{eqnarray}\label{eqn:StdMap}
S_1
\left(
\begin{array}{c}
x \\
y
\end{array}
\right)
=
\left(
\begin{array}{c}
x+y \\
y+ \sin x
\end{array}
\right)
\pmod{2\pi}.
\end{eqnarray}
The left-hand panel of Fig.~\ref{fig:StdMap_global} shows the
trajectories starting at a variety of different initial conditions
plotted in different colors. The shaded set is a large invariant chaotic set
with chaotic behavior, but many other invariant sets consist of
one or more topological circles, on which the system has quasiperiodic
behavior. For example, initial conditions $(\pi,1.65)$ leads to chaos
while $(\pi,1.5)$ leads to a quasiperiodic trajectory. As is clearly
the case here, one-dimensional
quasiperiodic sets often occur in families for non-linear processes,
structured like the rings of an onion. There are typically narrow bands of chaos between quasiperiodic onion rings. Usually these inner rings are differentiable
images of the $n$-torus, Yamaguchi and Tanikawa~\cite{StdMap1} and Chow et. al. in \cite{1halfDegree}
showed that the outermost limit (the onion's skin, to continue the analogy) will
still be a torus, but may not be differentiable. We have computed the rotation number 
for the standard map orbit shown in the right panel Fig.~\ref{fig:StdMap_global} using
quadruple precision. 
Its 30 digit 
value is given in Table~\ref{table:summary}.

{\bf The forced Van der Pol oscillator.} Fig.~\ref{fig:vdP_global} shows orbits for the time-$2\pi/0.83$ map of the
following periodically forced Van der Pol oscillator with nonlinear
damping~\cite{vdPol_forced}
\begin{equation}\label{vdPol_forced}
\frac{d^2 x}{dt^2}-0.2 \left( 1-x^2 \right) \frac{dx}{dt}+20 x^3=F\sin \left(0.83 t \right),
\end{equation}
with a variety of choices of $F$. While the innermost orbit shown is a
chaotic attractor, the outer orbits are topological circles with
quasiperiodic behavior\footnote{As with the standard map, we have
specified all non-essential parameters rather than stating the most
general form of the Van der Pol equation.}. Our computed the rotation number 
for the three orbits $F=15,25,35$ are given in Table~\ref{table:summary}. 

{\bf A two-dimensional torus map.} 
We now introduce an example of a two-dimensional quasiperiodic torus map on $\mathbb{T}^2$. This is a two-dimensional version of Arnold's family of 1D maps (see \cite{Arnold}), originally introduced in the following two papers~\cite{Grebogi:83,Grebogi:85}. The map is given by $(T_1,T_2)$ where 
\begin{eqnarray*}\label{eqn:2D_map}
T_1(x,y) = \left[ x+ \omega_1 + \frac{\epsilon}{2 \pi} P_1(x,y) ) \right] \pmod 1\\
T_2(x,y) = \left[ y+ \omega_2 + \frac{\epsilon}{2 \pi} P_2(x,y) \right] \pmod 1,
\end{eqnarray*}
and $P_i(x,y), i=1,2$ are periodic functions with period one in both variables, defined by:
\[
P_i(x,y) = \sum_{j=1}^4 a_{i,j} \sin (2 \pi \alpha_{i,j}), \mbox{ with } \alpha_{i,j}= r_j x + s_j y +b_{i,j}.
\]
The values of all coefficients are given in Table~\ref{table}. 
\begin{table}[t]
\begin{center}
{\begin{tabular}{ |c | c| }
\hline
Coefficient & Value \\
\hline
$\epsilon$ & $ 0.4234823$ \\[1pt]
$\omega_1$ & $ 0.71151134457776362264681206697006238 $ \\[2pt]
$\omega_2$ & $ 0.87735009811261456100917086672849971 $ \\[1pt]
$a_{1,j}$ & $(-0.268,-0.9106,0.3,-0.04)$ \\[1pt]
$a_{2,j}$ & $(0.08,-0.56,0.947,-0.4003)$ \\[1pt]
$b_{1,j}$ & $(0.985,0504,0.947,0.2334)$ \\[1pt]
$b_{2,j}$ & $(0.99,0.33,0.29,0.155)$ \\[1pt]
$r_j$ & $(1,0,1,0)$ \\[1pt]
$s_j$ & $(0,1,1,-1)$ \\[1pt]
Computed $\rho_1$ & $0.718053759982066107095244936117$ \\[1pt]
Computed $\rho_2$ & $0.885304666596099792113366824157$ \\[1pt]
\hline
\end{tabular}}
\end{center}
\caption{\label{table}
{\bf Coefficients for the torus map.} All values are used in quadruple precision, 
but in this table the repeated zeros on the end of the number are suppressed.}
\end{table}
This choice of this function is based on~\cite{Grebogi:83,
Grebogi:85}. Both papers use the same form of equation, though
the constants are close to but not precisely the same as the ones used previously. 
This is fitting with the point of view advocated in by these papers, in that the constants should be randomly chosen. Since we
are using higher precision, we have chosen constants
that are irrational to the level of our precision. The forward orbit is
dense on the torus, and the map is a nonlinear which exhibits two-dimensional quasiperiodic behavior.

Fig.~\ref{fig:2DMap_overview}a depicts iterates of the orbit, indicating that it is
dense in the torus. In order to verify that this map is indeed quasiperiodic,
we used a Birkhoff average to compute the two Lyapunov exponents to verify that both are 
zero ({\em cf.} Fig.~\ref{fig:2DMap_overview}b). In terms of method, this is just a matter of changing the function $f$ used in $\Q_N$ in Eq. \ref{eqn:QN}. 
Likewise, finding rotation numbers in two dimensions uses the same technique as in the one-dimensional case ({\em cf.} Fig.~\ref{fig:2DMap_overview}c).
In all of our calculations, the computation is significantly longer than in one dimension in order
to get the same accuracy, perhaps because in two dimensions, coverage of dense orbit varies like the square 
of the side length of the domain. 
\begin{figure}[t]
\centering
\subfigure[ ]{\includegraphics[height= .3\textwidth]{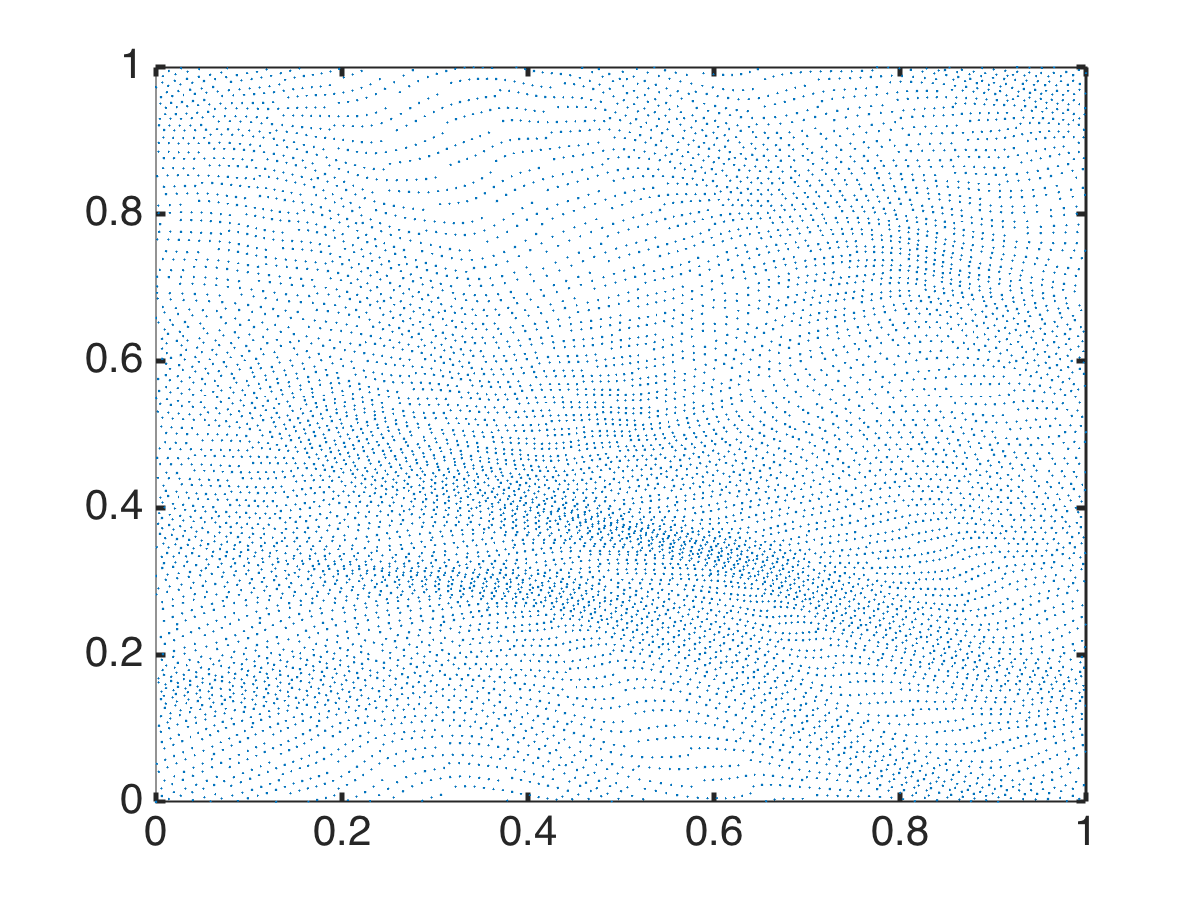}}
\subfigure[ ]{\includegraphics[height= .3\textwidth]{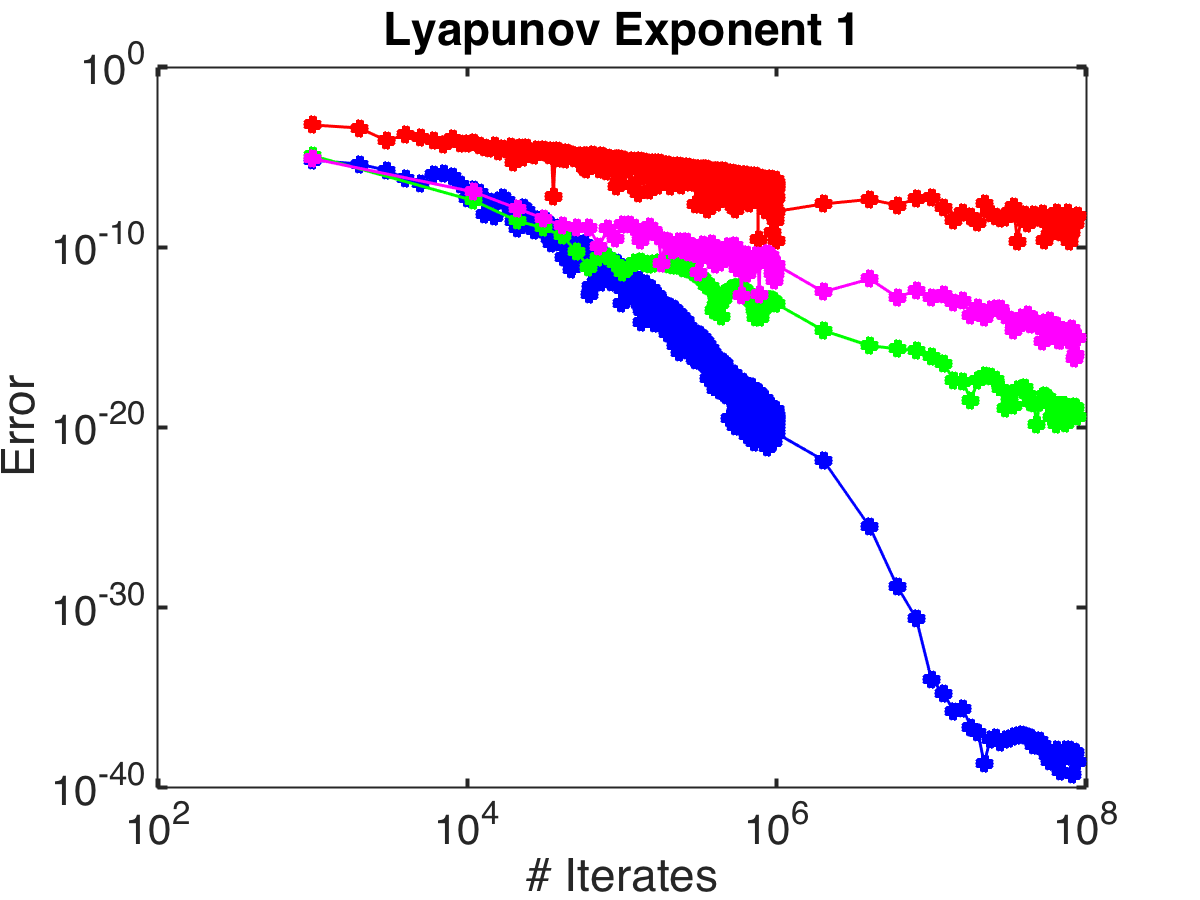}}
\subfigure[ ]{\includegraphics[height= .3\textwidth]{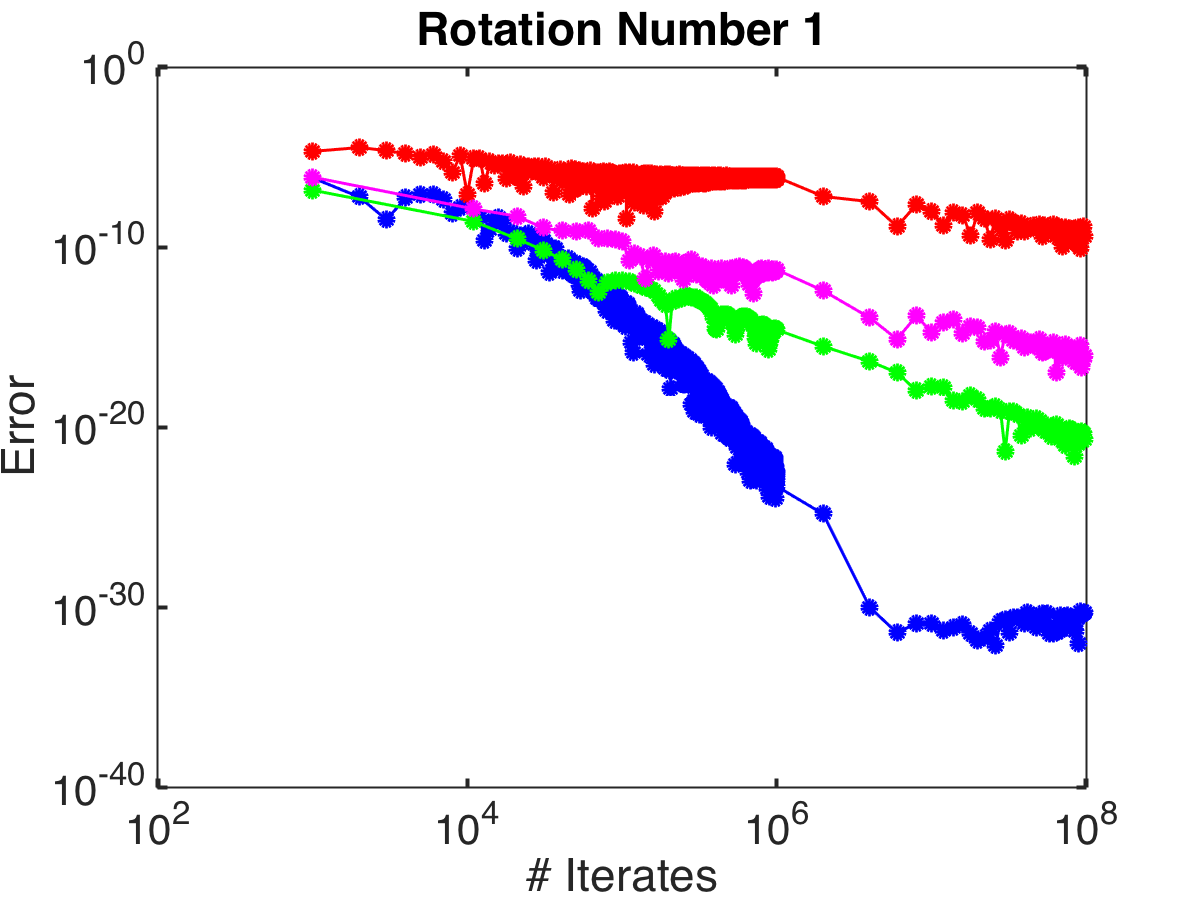}}
\caption{\textbf{Two-dimensional torus map.} 
Fig. (a) shows the first $10^{4}$ iterates of an orbit for the two-dimensional quasiperiodic 
torus map.
The orbit appears to be dense, indicating quasiperiodicity. This is
confirmed by computing the two Lyapunov exponents of the orbit using
Birkhoff averaging and finding them to be 0. The convergence of this
computation for one of the two Lyapunov exponents is shown in blue in
(b). The highest to lowest curves show the convergence rates resp.
for the first three weighting functions given in~(\ref{eqn:weights}).
Fig. (c) shows the convergence rate for the first rotation number for
the four different weighting functions. 
}
\label{fig:2DMap_overview}
\end{figure}
\begin{figure}[t]
\centering
\subfigure[ ]{\includegraphics[height=.4\textwidth] {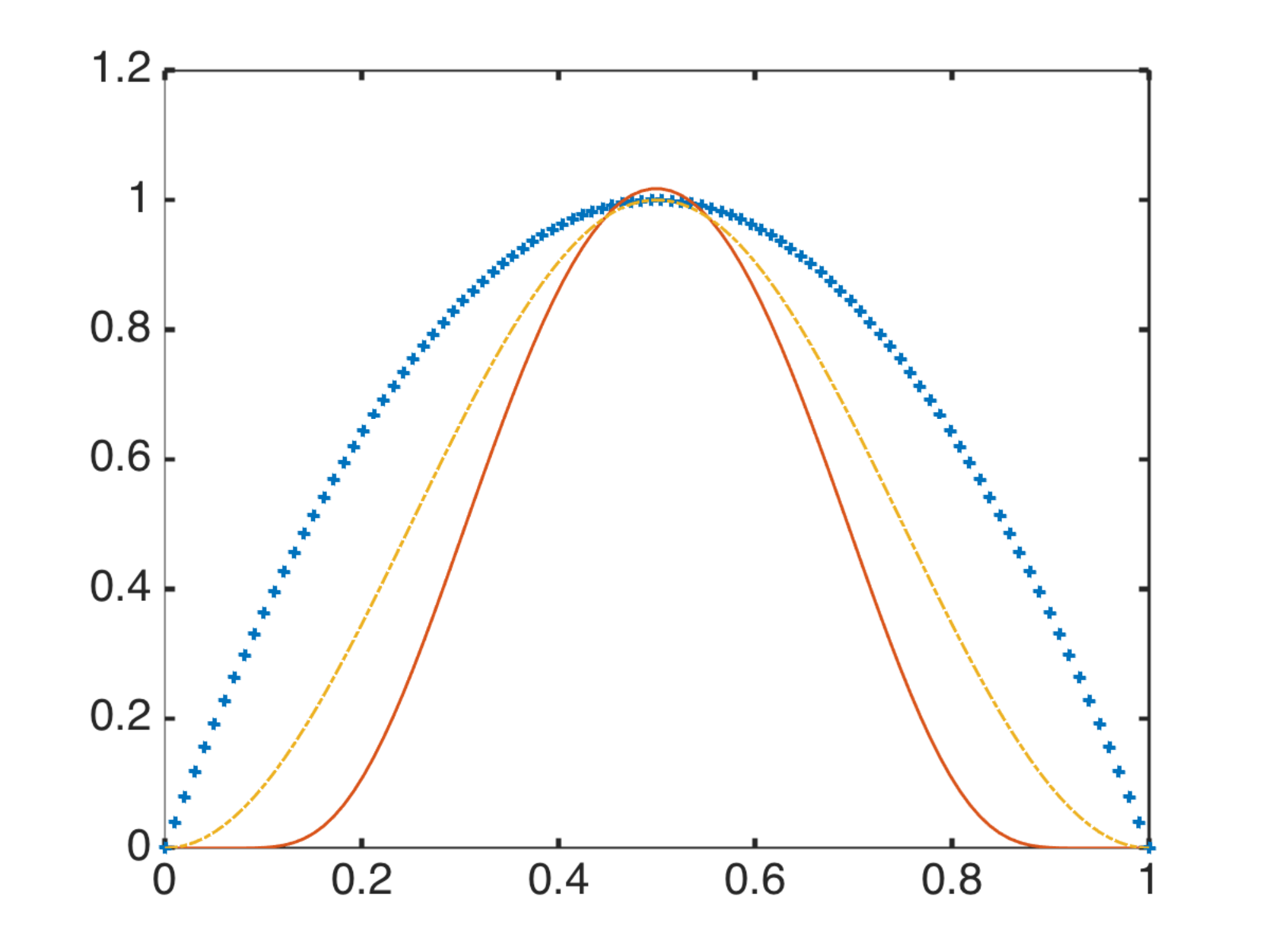}}
\subfigure[ ]{\includegraphics[height=.4\textwidth] {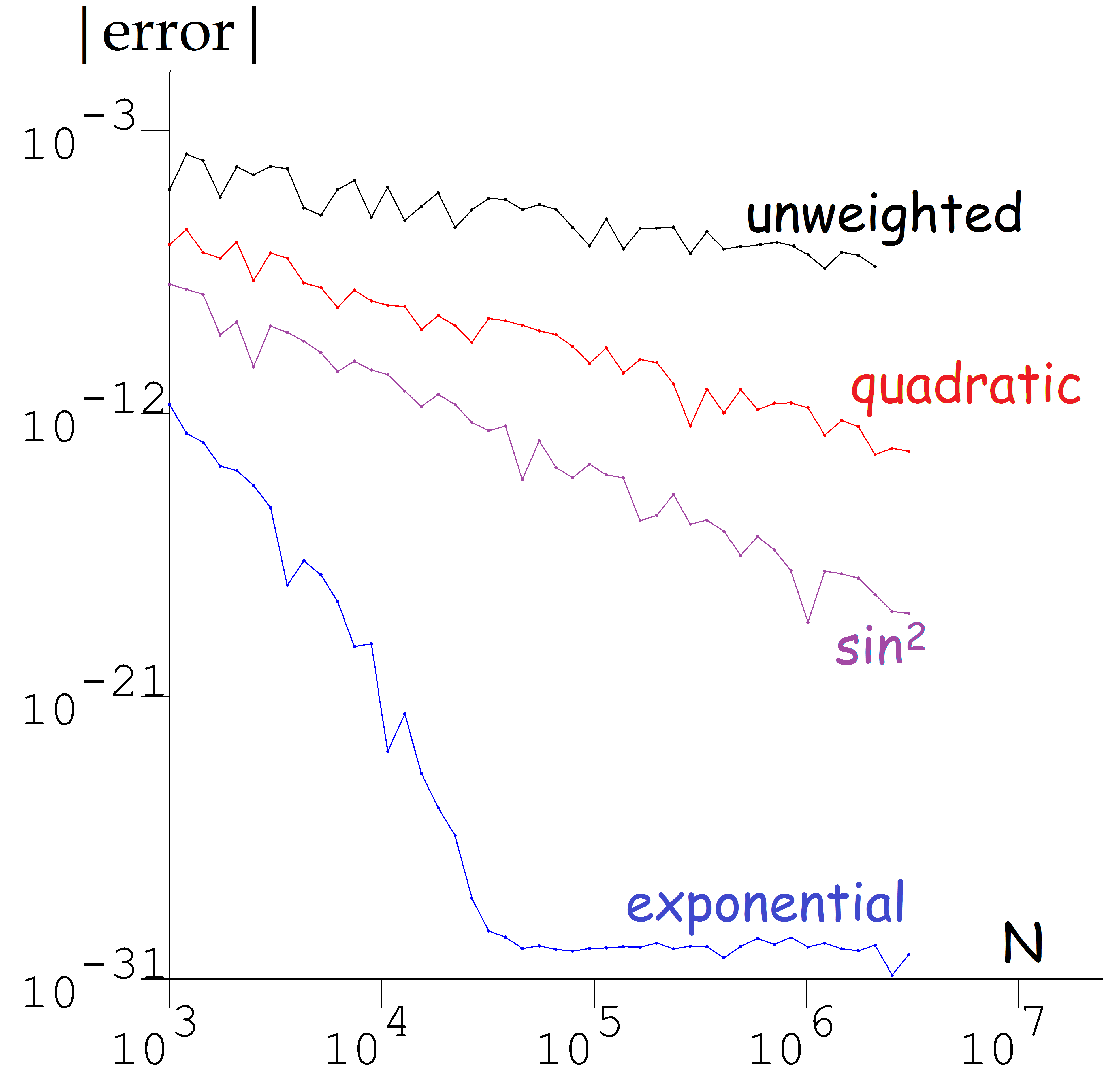}}
\caption{\textbf{Rate of convergence for different weight functions}. \label{fig:mid_circle}
Left is a plot of the three non-constant weighting functions from~(\ref{eqn:weights})
(quad blue, $\sin^2$ yellow, $\exp$ red). 
Since only the shape matters, they have been
rescaled so that each has a peak of approximately one.
For a given $w$ and a given number of iterates $N$, the rotation number
$\hat{\rho}$ approximation is calculated for $B_1$ of the restricted three-body problem, the error of the calculation is the
difference $|\rho-\hat{\rho}|$. The figure shows the convergence rate of
this error as a function of $N$. The exponential weight function is seen to be the
best method.
The error
cannot be reduced below $10^{-32}$ because that is the limit of
quadruple precision. }
\end{figure}
\begin{figure}[t]
\centering
\subfigure[ ]{\includegraphics[width = .46\textwidth]{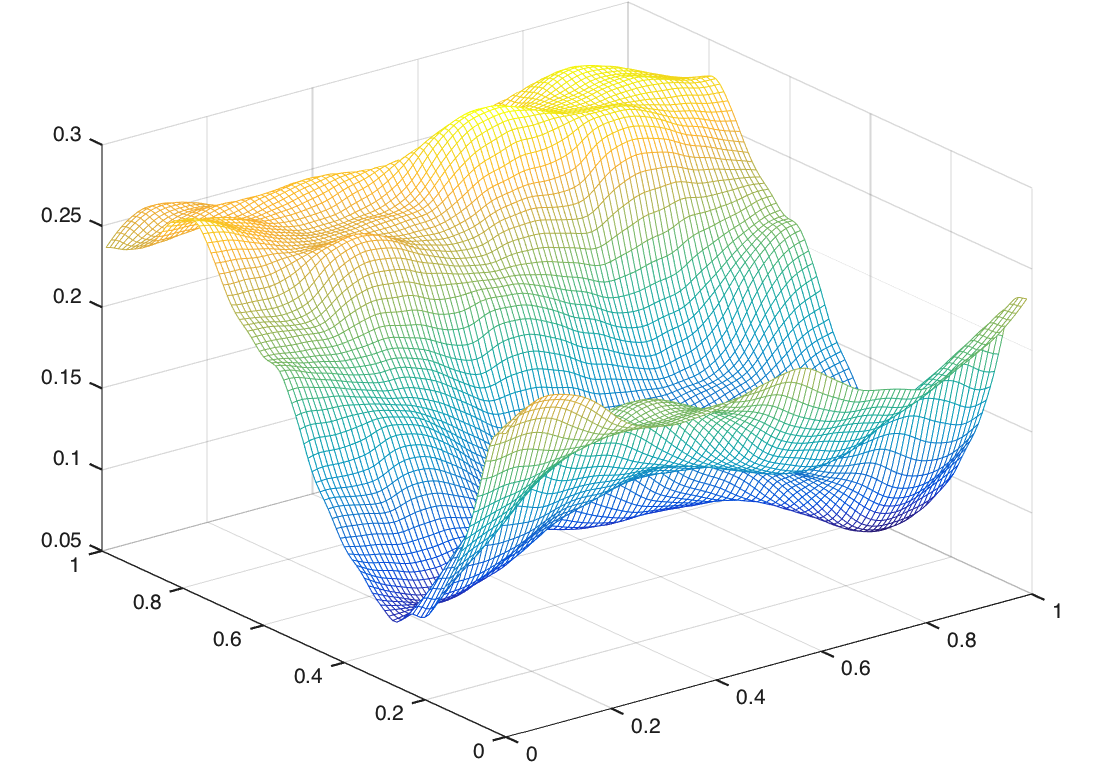}}
\subfigure[ ]{\includegraphics[width = .46\textwidth]{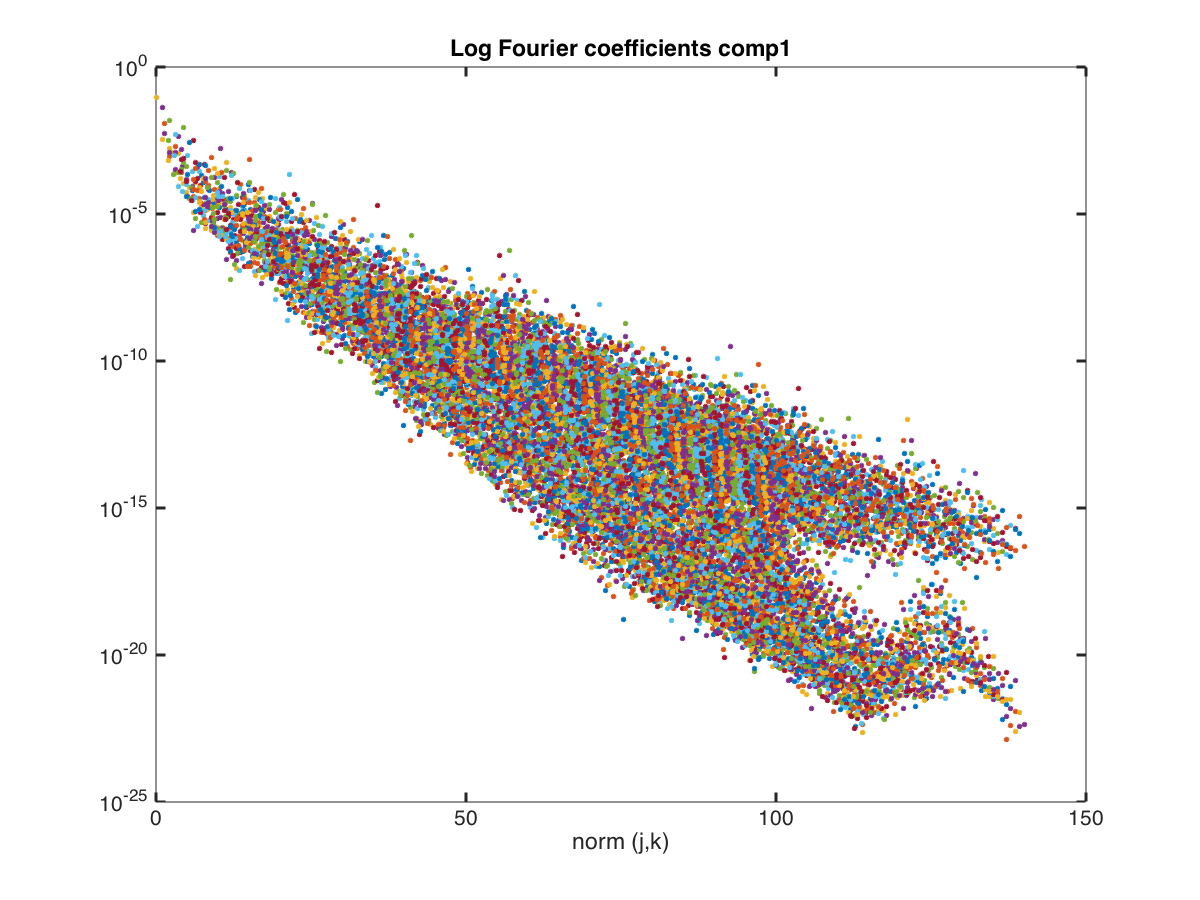}}
\caption{\textbf{Conjugacy for the torus.} \label{2DTorus_results}
Fig. (a) depicts the reconstruction of the periodic part $g$ (see \ref{eqn:g_conjufacy}) of the first component of the 
conjugacy function for the torus map. The second conjugacy function is similar but not depicted here. 
Fig. (b) shows the decay of the Fourier coefficients for this component of the conjugacy function on the
log-linear scale. 
}
\end{figure}

{\bf Convergence rate.}
In order to explain why the convergence of our method is so good, we introduce four different
possible values for the weighting function $w$, depicted in Fig.~\ref{fig:mid_circle}a, 
and compare the convergence results for computing 
the rotation number for each of these choices of $w$. 
\begin{eqnarray}\label{eqn:weights}
w_{equal}(t) & = & 1\mbox{ (Birkhoff's choice)}\\
w_{quad}(t) & = & t (1-t)\\
w_{(\sin^2)}(t) & = & \sin^2(\pi t) \\
w_{\exp}(t) & = & \exp\left({-1}/{(t(1-t))}\right).
\end{eqnarray}
If we compute with the first choice of $w$, we recover the a truncated
series in the definition of the Birkhoff average. To estimate the error, we 
expect the difference $f(x_{N+1})-f(x_{N})$ to be of order one, implying that 
\[\Q_{w_{equal},N+1} - \Q_{w_{equal},N} \sim 1/N.\]
The choice of a particular starting point also creates a similar uncertainty of order $1/N$.
For all but the first
choice, $w$ is always positive between $0$ and $1$ but vanishes as $t$
approaches $0$ and $1$. In addition, going down the list, increasing
number of derivatives of $w$ vanish for $t \to 0$ and $t\to 1$, with all derivatives of $w_{\exp}$ 
vanishing at $0$ and $1$. We thus expect the 
effect of the starting and endpoints to decay at the same rate as this number of 
vanishing derivatives. Indeed, we find that $w_{quad}$ corresponds approximately to order $1/N^2$ convergence,
$w_{(\sin^2)}$ to $1/N^3$ convergence, and $w_{\exp}$ to convergence faster than any polynomial. 
Figs.~\ref{fig:mid_circle}b and ~\ref{fig:2DMap_overview}bc show this effect.

{\bf Related methods}. See
\cite{seara:villanueva:06,luque:villanueva:14} for references to earlier
methods for computing rotation numbers. In 
\cite{seara:villanueva:06,luque:villanueva:14}, A. Luque and J.
Villanueva develop fast methods for obtaining rotation numbers for
analytic functions on a quasiperiodic torus, sometimes with
quasiperiodic forcing with several rotation numbers. The paper
\cite{luque:villanueva:14} examines a smooth function f on a
quasiperiodic torus. Let $f_n$ denote the value of f at the n-th
trajectory point. From this sequence they can obtain the rotation number
with error satisfying $|error| \le C_p N^{-p}$ for any $p$ where $C_p$
is a constant. The method of computation depends on $p$ and as $p$
increases the computational complexity increases for fixed $N$. If
$T(p,N)$ is their computation time, it appears that
$T(p,N)/N\rightarrow\infty$ as $p\rightarrow\infty$. In comparison, computation time for our weighted Birkhoff average is
simply proportional to $N$ since it requires a sum of $N$ numbers. The paper gives one figure (Fig 6) from which the rate of convergence can be computed, namely a restricted three-body problem. Their rotation-rate error is proportional to $N^{-3.5}$ and is $\approx 10^{-18}$ at $N = 2^{21}$.

Several variants of the Newton's
method have been employed to determine quasiperiodic trajectories in
different settings. In \cite{NumericQuasi3} the monodromy variant of
Newton’s method was applied to locate periodic or quasi-periodic
relative satellite motion. 
A PDE-based approach was
taken in \cite{NumericQuasi4}, where the authors defined an invariance
equation which involves partial derivatives. The invariant tori are
then computed using finite element methods. See also Chapter 2,
\cite{NumericQuasi4} for more references on the numerical computation of
invariant tori.

\subsection{Fourier coefficients and conjugacy reconstruction}\label{sec:Fourier}

For a quasiperiodic curve as shown in Fig. \ref{fig:results_3B}a, there are two approaches to representing the curve. Firstly, we can write the coordinates $(X,Y)$ as a function of $\theta\in S^1$, or secondly, we can reduce the dimension and represent the points on the curve by an angle $\phi\in S^1$, that is, $\phi(X(\theta),Y(\theta)$, which is also $h(\theta)=\theta+g(\theta)$. We have shown $g$ in Fig \ref{fig:results_3B}b and the exponential decay of the norm of the Fourier coefficients in Fig. \ref{fig:results_3B}c. To limit the number of graphs in this paper, we have only created the Fourier series for the periodic part $g(\theta)$

Given a continuous periodic map $f:S^1\rightarrow\mathbb{R}$, where $S^1$ is a circle (or 
one-torus), 
the Fourier sine and cosine representation of $f$ is the following.
\begin{equation}\label{eqn:sine_cosine_expansion}
\mbox{For every }t\in S^1,\ f(t)=\frac{b_0}{2}+ \sum_{k=1}^\infty b_k \cos(2k\pi t)+ \sum_{k=0}^\infty c_k \sin(2k\pi t)
\end{equation}
where the coefficients $b_k$ and $c_k$ are given by the following formulas.
\begin{equation}\label{eqn:a_n}
b_k = 2 \int_{\theta\in S^1} f(\theta) \cos(2 k \pi\theta) \; d\theta,
\end{equation}
\begin{equation}\label{eqn:b_n}
c_k = 2 \int_{\theta\in S^1} f(\theta) \sin(2 k \pi\theta)\; d\theta.
\end{equation}

To be able to use the fast Fourier transform, $2^M$ equally spaced points on the circle are required. If we only have access to an ergodic orbit $(x_n)$ on the circle, then we cannot use the fast Fourier transform as we only have the function values $f(x_n)$ along a quasiperiodic trajectory, and a rotation number $\rho$. So instead, we obtain these coefficients using a weighted Birkhoff average on a trajectory
$(x_n)$ by applying the functional $\Q_{w,N}$. For $k=0$, we find $a_0$ by applying $\Q$ to the function $1$. 
For $k>0$, we find $b_k$ and $c_k$ as follows. 
\begin{equation}\label{eqn:a_n_from_Q}
b_k = \Q_{w,N} (f(\theta) \cos(2 k \pi\theta)) = \sum_{n=0}^N f(x_n) \cos(2k\pi n\rho) \hat{w}_{n,N}.
\end{equation}
\begin{equation}\label{eqn:b_n_from_Q}
c_k = \Q_{w,N} (f(\theta) \sin (2 k \pi\theta))= \sum_{n=0}^N f(x_n) \sin(2k\pi n\rho) \hat{w}_{n,N}.
\end{equation}
By specifying that $\theta_0 = 0$, our computation of rotation number $\rho$ provides all iterates : $\theta_n = n \rho$. Using the Fourier coefficients, we can 
thus reconstruct the periodic part of the change of variables function $g$ (see Eq. \ref{eqn:g_conjufacy}). 
This is depicted for the restricted three-body problem in Fig.~\ref{fig:3B2}, 
for the standard map in Fig.~\ref{fig:StdMap}, for the forced van der Pol equation in 
Fig.~\ref{fig:vdP_global}. 
In all three one-dimensional cases, we depict $\sqrt{b_k^2 + c_k^2}$ as a function of $k$. 
Our main observation is that the Fourier
coefficients decay exponentially; that is, for some positive numbers $\alpha$
and $\beta$, in dimension one, the Fourier coefficients $b_k$ and $c_k$ satisfy
\begin{equation}\label{eqn:Fourier_decay}
\sqrt{|b_k|^2+|c_k|^2} \le \alpha e^{- \beta |k|} \mbox{ for all }k\in\mathbb{Z}. 
\end{equation}
This is characteristic of analytic functions. We therefore state that all of the conjugacy functions 
that we computed in our examples are effectively analytic,
``effectively'' meaning within the precision of our quadruple precision
numerics. 

In two dimensions, the computation of Fourier coefficients is similar,
but instead of only having one set of cosine and sine functions, for
each $(j,k)$, we have two linearly independent sets of complex-valued functions, 
where $i = \sqrt{-1}$: 
\[ e^{i (j x + k y)} \mbox{ and } e^{i (j x - k y)}. \]
We define $a_{j,k}$ and $b_{j,k}$
to be the complex-valued coefficients corresponding to each of these
functions. The reconstructed conjugacy function and decay of
coefficients for the two-dimensional torus is depicted in
Fig.~\ref{2DTorus_results}. The decay of coefficients shows $\sqrt{j^2 +
k^2}$ on the horizontal axis, and $|b_{j,k}|$ and $|c_{j,k}|$ on the
vertical axis, where both of these coefficients are complex, meaning
that $|\cdot|$ represents the modulus. Again here, the coefficients
decay exponentially, though the decay of coefficients is considerable
slower in two dimensions due to the added dimension. The data looks 
quite a lot more crowded in this case, since there are many different 
values of $(j,k)$ such that the values of $\sqrt{j^2 + k^2}$ are identical or
very close. In addition, the two sets of
coefficients $b_{j,k}$ and $c_{j,k}$ generally converge at different exponential
rates. This is why there is a strange looking set of two different
clouds of data in Fig.~\ref{2DTorus_results}b. While more information on
the difference between these coefficients is gained by interactively
viewing the data in three dimensions, we have not been able to find a
satisfactory static flat projection of this data. We feel that in a still image, 
the data cloud shown conveys the maximum information. 

\textbf{Accuracy of the calculation of Fourier coefficients}. Our method of calculation of the Fourier coefficients is dependent on the knowledge of the rotation vector $\vec\rho$, or an approximation $\hat{\vec\rho}$. For the rest of the chapter, $m\in\mathbb{N}$ is some fixed integer and $M$ is some integer satisfying $M>d+m(d+\beta)$, where $\beta$ is the Diophantine class of $\vec\rho$. Let $\Delta\vec\rho$ be the error in the approximation of $\rho$, that is,
\[\Delta\vec\rho:=\hat{\vec\rho}-\vec\rho.\]
We are interested in knowing how the error in the calculation of Fourier coefficients $b_k, c_k$ depend on $\Delta\vec\rho$. Let for every $k\in\mathbb{Z}^d$, \boldmath $f_k(\theta)$ \unboldmath $:=e^{i 2\pi k\cdot\theta}$. Then every periodic function $g$ has the complex Fourier series representation
\[g(\theta)=\underset{k\in\mathbb{Z}^d}{\Sigma}a_k f_k,\mbox{ where }a_k=\int_{\torus}g(\theta)f_{-k}(\theta)\]
Therefore, $a_k$ can be approximated as
\[\hat a_k=\Q_{N,\hat{\vec\rho}}(g(\theta)f_{-k}(\theta))=\sum\limits_{n=0}^{N-1}g(n\vec\rho)e^{-i 2\pi n k\cdot \hat{\vec\rho}}\hat{w}_{n,N}.\]

\begin{corollary}\label{cor:error_with_rho}
If $N\|k\cdot\Delta\vec\rho\|<<1$, then for each $m\in\mathbb{N}$, there exists a constant $C_{g,w,m}>0$ such that 
\[\frac{|\hat a_k(g)-a_k(g)|}{|a_k(g)|}\leq \pi N |k\cdot \Delta\vec\rho|+C_{g,w,m}N^{-m}.\]
\end{corollary}
\begin{proof}
We begin by obtaining bounds on the error in calculating the Fourier coefficients of pure exponentials. Note that $a_k(f_l)$ equals $1$ if $l=k$, equals $0$ if $l\neq k$. We will estimate the error $|\hat a_k(f_l)-a_k(f_l)|$ in each case.
\[|\hat a_k(f_k)=\Q_{N,\hat{\vec\rho}}(f_k(\theta)f_{-k}(\theta))| =\sum\limits_{n=0}^{N-1}e^{-i 2\pi n k\cdot \Delta\vec\rho}\hat{w}_{m,N}.\]
Then if $N\|k\cdot\Delta\vec\rho\|<<1$, then using the approximation $e^{i\alpha}-1\approx i\alpha$, one can say that
\[\begin{split}
|\hat a_k(f_k)-a_k(f_k)|&=|\hat a_k(f_k)-1|\\
&=\left|\sum\limits_{n=0}^{N-1}(e^{-i 2\pi n k\cdot \Delta\vec\rho}-1)\hat{w}_{m,N}\right|\\
&\leq\sum\limits_{n=0}^{N-1}|e^{-i 2\pi n k\cdot \Delta\vec\rho}-1|\hat{w}_{m,N}\\
&\leq\sum\limits_{n=0}^{N-1}2\pi n k\cdot |\Delta\vec\rho|\hat{w}_{m,N}\\
&=\pi N |k\cdot \Delta\vec\rho|.\\
\end{split}\]
Since $\vec\rho$ is Diophantine with Diophantine class $\beta\geq 0$, for every $\alpha\in\mathbb{R}$, there exists a constant $C_{\vec\rho}>0$ such that for all $l\in\mathbb{Z}^d$ with $\|l\|$ sufficiently big, one has
\begin{equation}\label{eqn:shifted_Dioph}
|\exp(i2\pi(l\cdot\vec\rho-\alpha)-1)|=|\exp(i2\pi(l\cdot\vec\rho-\alpha)-\exp(i2\pi \alpha))|\geq C_{\vec\rho}\|l\|^{-(d+\beta)}
\end{equation}
By Eq. 14 from \cite{Das-Yorke}, for every $l\neq k$ there exists a constant $C_{w,m}>0$ depending on $w$ and $m$ such that for every $\alpha\in\mathbb{R}$
\begin{equation}\label{eqn:DY_14}
\left|\sum\limits_{n=0}^{N-1}\exp(-i 2\pi n \alpha)\hat{w}_{m,N}\right| \leq C_{w,m}N^{-m}|e^{i 2\pi \alpha}-1|^{-m}
\end{equation}
Combining Eqs. \ref{eqn:shifted_Dioph} and \ref{eqn:DY_14}, we get
\[\begin{split}
|\hat a_k(f_l)-a_k(f_l)|&=|\hat a_k(f_l)|\\
&=\left|\sum\limits_{n=0}^{N-1}\exp(-i 2\pi n (l\cdot \vec\rho-k\cdot \hat{\vec\rho}))\hat{w}_{m,N}\right|\\
&\leq C_{w,m}N^{-m}|\exp(i 2\pi (l\cdot \vec\rho-k\cdot \hat{\vec\rho}))-1|^{-m}\\
&\leq C_{w,m}C_{\vec\rho}N^{-m}\|l\|^{m(d+\beta)}\\
\end{split}\] 

Since $f\in C^{M}$, there exists $C_{g,M}>0$ depending on $g$ and $M$ such that for each $l\neq \vec 0$, $|a_l|\leq C_{f,M} \|l\|^{-M}$. Therefore, in a manner similar to the derivation of Eq. 15 from \cite{Das-Yorke}, we can write
\[\begin{split}
|\hat a_k(g)-a_k(g)|&=\left|\underset{l\in\mathbb{Z}^d}{\Sigma}a_l(a_k(f_l)-a_k(f_l))\right|\\
&\leq |a_k(a_k(f_k)-a_k(f_k))|+\left|\underset{l\neq k}{\Sigma}a_l(a_k(f_l)-a_k(f_l))\right|\\
&\leq |a_k|\pi N |k\cdot \Delta\vec\rho|+\underset{l\neq k}{\Sigma}|a_l(a_k(f_l)-a_k(f_l))|\\
&\leq |a_k|\pi N |k\cdot \Delta\vec\rho|+C_{w,m}C_{\vec\rho}N^{-m}\underset{l\neq k}{\Sigma}|a_l|\|l\|^{m(d+\beta)}\\
&\leq |a_k|\pi N |k\cdot \Delta\vec\rho|+C_{w,m}C_{\vec\rho}C_{f,M}N^{-m}\underset{l\neq k}{\Sigma}\|l\|^{-M}\|l\|^{m(d+\beta)}\\
&= |a_k|\pi N |k\cdot \Delta\vec\rho|+C_{w,m}C_{\vec\rho}C_{f,M}N^{-m}\underset{l\neq k}{\Sigma}\|l\|^{-(M-m(d+\beta))}\\
\end{split}\]
Since $M>d+m(d+\beta)$, the sum $\underset{l\neq k}{\Sigma}\|l\|^{-(M-m(d+\beta))}<\infty$. The statement of the corollary now follows. \qed
\end{proof}

\clearpage

\bibliographystyle{unsrt}
\bibliography{Weighted_calc_bibliography,1D_bibliography}
\end{document}